\documentclass[11pt,reqno]{amsart}
\usepackage{graphicx}
\usepackage{amssymb}
\usepackage[usenames,dvipsnames]{color}
\newtheorem{theorem}{Theorem}[section]
\newtheorem{lemma}{Lemma}[section]
\newtheorem{proposition}{Proposition}[section]
\newtheorem{definition}{Definition}[section]

\newtheorem{remark}{Remark}[section]

\numberwithin{equation}{section}

\def\essinf{\mathop{\rm ess~inf}}

\def\N{\mathbb{N}}

\def\R{\mathbb{R}}

\def\Z{\mathbb{Z}}

\def\ra{\rightarrow}

\def\pa{\partial}

\def\al{\alpha}

\def\ep{\epsilon}

\begin{document}

\title[Lattice KPP equations in time heterogeneous media]{Spreading speeds and transition fronts of lattice KPP equations in time heterogeneous media}

\author{Feng Cao}
\address{Department of Mathematics, Nanjing University of Aeronautics and Astronautics, Nanjing, Jiangsu 210016, P. R. China}
\email{fcao@nuaa.edu.cn}
\thanks{Research of F. Cao was supported by the Fundamental Research Funds for the Central Universities No. NZ2014104.}

\author{Wenxian Shen}
\address{Department of Mathematics and Statistics, Auburn University, Auburn, AL 36849, USA}
%\curraddr{Department of Mathematics and Statistics,
%Case Western Reserve University, Cleveland, Ohio 43403}
\email{wenxish@auburn.edu}
%\thanks{The author was supported in part by NSF Grant \#000000.}

%    General info

%\date{January 1, 2001 and, in revised form, June 22, 2001.}



\begin{abstract}
The current paper is devoted to the study of spreading speeds and transition fronts of lattice KPP equations in time heterogeneous media. We first
prove the existence, uniqueness, and stability of spatially homogeneous entire positive solutions. Next, we establish lower and upper bounds of the (generalized) spreading speed intervals. Then, by constructing appropriate  sub-solutions and super-solutions, we show
the existence and continuity of  transition fronts with given front position functions. Also, we prove the existence of some kind of critical front.

\end{abstract}

\maketitle

%\tableofcontents

%%%%%%%%%%%%%%%%%%%%%%%%%%%%%%%%%%%%%%%%%%%%%%%%%%%%%%%%%%%%%%%%%%%%%%%%
%\footnote{Here is an example of a footnote.}%
%%%%%%%%%%%%%%%%%%%%%%%%%%%%%%%%%%%%%%%%%%%%%%%%%%%%%%%%%%%%%%%%%%%%%%%%

\section{Introduction}

The current paper deals with spatial spreading speeds and transition fronts for the following discrete Fisher-KPP equation in time heterogeneous media
\begin{equation}\label{main-eqn}
\dot{u}_{i}(t)=u_{i+1}(t)-2u_{i}(t)+u_{i-1}(t)+u_i(t)f(t,u_{i}(t)),i\in \Z,
\end{equation}
where $f(t,u)$ is of monostable type. More precisely, we  assume

\medskip

\noindent{\bf (H0)} $f(t,u)$ is locally H\"older continuous in $t\in\R$, Lipschitz continuous in $u\in\R$, and continuously differentiable in $u$ for $u\ge 0$.

\medskip
\noindent {\bf (H1)} For each $u\in\R$, $f(\cdot,u)\in L^{\infty}(\R)$;  $f(t,u)<0$ for $u\ge M_0$ and some $M_0>0$;   $f_u(t,u)<0$ for $u\ge 0$;  and
\begin{equation}
\label{assumption-eq}
\liminf_{t-s\to \infty}\frac{1}{t-s}\int_s^t f(\tau,0)d\tau>0.
\end{equation}

\medskip

\noindent {\bf (H2)} There are $0<\tilde m_0< \tilde M_0$ such that $f(t, 0)-\tilde M_0 u\le f(t,u)\le f(t,0)-\tilde m_0 u$ for $u\geq 0$.

\medskip

Let
$$
l^\infty(\Z)=\{u=\{u_i\}_{i\in \Z}:\sup \limits_{i \in\Z}|u_i|<\infty\}
$$
with norm $\|u\|=\|u\|_\infty=\sup_{i\in\Z}|u_i|$. By (H0), for any given $u^0\in l^\infty(\Z)$ and $s\in\R$, \eqref{main-eqn} has a unique (local) solution
$u(t;s,u^0)=\{u_i(t;s,u^0)\}_{i\in\Z}$ with $u(s;s,u^0)=u^0$. By (H1),  if $u^0_i\geq0$
for all $i\in \Z$, then $u(t;s,u^0)=\{u_i(t;s,u^0)\}_{i\in\Z}$ exists for all $t\ge s$ and $u_i(t;s,u^0)\geq0$ for all $i\in\Z$ and $t\geq s$ (see Proposition \ref{comparison}).

Equation (\ref{main-eqn}) is used to model the  population dynamics of species living in patchy environments in biology and ecology (see, for example, \cite{ShKa97, ShSw90}).
The following two equations are spatially continuous counterparts of \eqref{main-eqn},
\begin{equation}\label{eqn-con}
u_{t}=u_{xx}+uf(t,u)
\end{equation}
and
\begin{equation}
\label{eqn-nonlocal}
u_t(x,t)=\int_{\R}\kappa(y-x)u(y,t)dy-u(x,t)+u f(t,u),
\end{equation}
where $\kappa(\cdot)$ is a nonnegative smooth function with compact support and $\int_{\R}\kappa(z)dz=1$.
\eqref{eqn-con} is widely used to model the population dynamics of species when the movement or internal dispersal of the organisms
occurs between adjacent locations randomly in  spatially continuous media, and \eqref{eqn-nonlocal} is often used to model the population
 dynamics of species when the movement or internal dispersal of the organisms occurs between adjacent as well as non-adjacent locations
in  spatially continuous media. The dispersal described by $u\mapsto u_{xx}$ in \eqref{eqn-con} is therefore referred to as {\it random dispersal}
and the dispersal  described by $u(x,t)\mapsto \int_{\R}\kappa(y-x)u(y,t)dy-u(x,t)$ in \eqref{eqn-nonlocal} is referred to as {\it nonlocal dispersal}
in literature.

 One of the central dynamical issues about \eqref{main-eqn}, \eqref{eqn-con}, and \eqref{eqn-nonlocal}
 is to know how a solution whose initial datum  is strictly positive, or  is nonnegative and has compact support, or is a front-like function evolves as time increases.
   For example, it is important to know how a solution $u(t;s,u^0)$ of (\ref{main-eqn}) evolves as $t$ increases, where $u^0$ is strictly positive (that is, $\inf_{i\in\Z}u_i^0>0$), or $u^0$ is nonnegative (that is, $u^0_i\geq 0$ for all $i\in\Z$)  and has compact support (that is, $\{i\in \Z:u^0_i>0\}$ is a bound subset of $\Z$), or $u^0$ is nonnegative and a front-like function (that is,
$$\sup_{i\geq I}u^0_i\ra 0,~~~~\inf_{i\leq -I}u^0_i\ra u_*>0,~~~(I\ra \infty)).$$
This is strongly related to the so called spatial spreading speeds and traveling wave solutions.
Pioneering works on these issues are due to Fisher \cite{Fish37} and Kolmogorov, Petrovsky and Piscunov \cite{KPP37}. They studied the existence of traveling wave solutions of \eqref{eqn-con} when $f(u)=1-u$, that is, solutions which can be written as $u(x,t)=\phi(x-ct)$ with $\phi(-\infty)=1$, $\phi(+\infty)=0$.   Fisher in \cite{Fish37} found traveling wave solutions $u(x,t)=\phi(x-ct)$
$(\phi(-\infty)=1,\phi(\infty)=0)$ of \eqref{eqn-con} with $f(u)=1-u$ of all speeds $c\geq 2$ and showed that there are no such traveling wave
solutions of slower speed.  Kolmogorov, Petrowsky, and Piscunov in \cite{KPP37} proved that for any nonnegative
solution $u(x,t)$ of \eqref{eqn-con} with $f(u)=1-u$, if at time $t=0$, $u$ is $1$ near $-\infty$ and $0$ near
$\infty$, then $\lim_{t\to \infty}u(t,ct)$ is $0$ if $c>2$ and $1$ if $c<2$ (that is, the population invades into
the region  with no initial population  with speed
$2$). $c_*:=2$ is therefore the minimal wave speed of \eqref{eqn-con} with $f(u)=1-u$ and is also called the spatial spreading
speed of \eqref{eqn-con} with $f(u)=1-u$. Thanks to the works \cite{Fish37} and \cite{KPP37}, \eqref{main-eqn}, \eqref{eqn-con}, and \eqref{eqn-nonlocal} with $f$ satisfying
(H1) and (H2) are called Fisher or KPP type equations.

Since the pioneering works by Fisher (\cite{Fish37}) and Kolmogorov, Petrowsky, Piscunov  (\cite{KPP37}), spatial spreading speeds and
traveling wave solutions of Fisher or KPP type evolution equations  in spatially and temporally homogeneous media or spatially and/or temporally
periodic media have been widely studied.
 The reader is referred to \cite{ArWe75, ArWe78,   BeHaNa1, BeHaNa2, BeHaRo, BeNa12, FrGa, HuSh09, Kam76, KoSh, LiZh1, LiZh2, Na09,  NoRuXi, NoXi, NRRZ12, Sa76, Sh10, Sh04,  Uch78, We02}, etc.,  for the study of Fisher or KPP type equations with random dispersal in homogeneous or periodic media,
to \cite{CaCh04, CoDu05, CDM13,  RaShZh, ShSh1,  ShZh10, ShZh12-1, ShZh12-2}, etc., for the study of Fisher or KPP type equations with nonlocal dispersal
in homogeneous or periodic media,
and  to \cite{ChFuGu06, ChGu02, ChGu03, FuGuSh02, GuHa06, GuWu09, HuZi94, WuZo97, ZiHaHu93}, etc., for the study of Fisher or KPP type equations
in spatially discrete homogeneous or periodic media.

The study of spatial spreading speeds and traveling wave solutions of KPP type equations with general time and/or space dependence is
more recent. Quite a few works have been carried out  toward the spatial spreading speeds and traveling wave solutions of
KPP type equations in non-periodic heterogeneous media. For example, in \cite{Sh04}, \cite{Sh11}, notions of random traveling wave solutions and generalized
traveling wave solutions are introduced for random KPP equations and
quite general time dependent KPP equations.
In \cite{BeHa07}, \cite{BeHa12}, a notion of generalized transition waves is introduced for KPP type equations with general space and
time dependence. Among others, the authors of \cite{NaRo12} proved the existence of generalized transition waves of \eqref{eqn-con} with general time dependent
KPP nonlinearity $f(t,u)$.  In \cite{BeNa12}, the authors studied spreading speeds of general spatially heterogeneous { including space periodic}  Fisher-KPP reaction diffusion
equations  (see also \cite{BeHaNa, Fre, FrGa, We02}, etc.).
Zlatos \cite{Zl12} established the existence of generalized transition waves of   spatially inhomogeneous Fisher-KPP reaction-diffusion equations { under some specific hypotheses (see (1.2)-(1.5) in \cite{Zl12})}.
In \cite{ShSh2}, the second author of the current paper together with Zhongwei Shen proved the existence, uniqueness, and stability of generalized transition waves of \eqref{eqn-nonlocal} with  time dependent KPP nonlinearity $f(t,u)$ { under some assumptions (see (H1), (H2) in \cite{ShSh2})}. The authors of \cite{LiZl14} obtained
the existence of generalized transition waves of  spatially heterogeneous KPP equations with nonlocal dispersal { under some specific assumptions
(see (J1), (J2), (F1)-(F3), (G1), (G2) in \cite{LiZl14})}. For spatially discrete KPP equations,
the work  \cite{Sh09} studied spatial spreading speeds of \eqref{main-eqn} with time recurrent KPP nonlinearity $f(t,u)$. However, there is little study on spatial spreading speeds and traveling wave solutions of spatially discrete KPP type equations  with general time and/or space dependence.

%In \cite{HuSh09}, Huang and Shen introduced a notion of spatial spreading speed interval for time almost periodic and space periodic monostable reaction-diffusion equations, which is the natural extension of the classical concept of the spreading speeds for time independent or periodic KPP models. And some fundamental properties of the spatial spreading speed interval for such monostable equations were established later (\cite{Sh10}). This theory for the spatial spreading can also be applied to spatially discrete and time recurrent KPP models of form (\ref{main-eqn}) (\cite{Sh09}).

In this paper, we are going to investigate spatial spreading speeds and generalized transition waves  for general time dependent KPP model (\ref{main-eqn}). Throughout this paper, we assume (H0)-(H2).

We first study the existence, uniqueness, and stability of spatially homogeneous entire solutions. A solution $u(t)=\{u_j(t)\}$ of \eqref{main-eqn} is called an {\it entire positive solution} if it is a solution of \eqref{main-eqn} for $t\in\R$ and $\inf_{t\in\R,i\in\Z}u_i(t)>0$. A solution $u(t)=\{u_i(t)\}$ of \eqref{main-eqn} is called {\it spatially homogeneous} if $u_i(t)=u_j(t)$ for all $i,j\in\Z$.
For given function  $t\mapsto u(t)\in l^\infty(\Z)$ and $c\in\R$, we define
$$\limsup_{|i|\ge ct, t\to\infty} u_i(t)=\limsup_{t\to\infty}\sup_{i\in\Z,|i|\ge ct} u_i(t),$$
$$\liminf_{|i|\le ct, t\to\infty} u_i(t)=\liminf_{t\to\infty}\inf_{i\in\Z,|i|\le ct} u_i(t),$$
and
$$\limsup_{|i|\le ct, t\to\infty} u_i(t)=\limsup_{t\to\infty}\sup_{i\in\Z,|i|\le ct} u_i(t).$$
 We prove

\begin{theorem}\label{stable-thm}
\begin{itemize}
\item[(1)]
There is a unique spatially homogeneous entire positive  solution $u^+(t)$ of \eqref{main-eqn} which is globally stable in the sense that for any $u^0\in l^\infty(\Z)$ with $\inf_{i\in \Z}u_i^0>0$,
$$\|u(t+s;s,u^0)-u^+(t+s)\|_\infty \to 0\quad {\rm as}\quad t\to\infty$$
uniformly in $s\in\R$.

\smallskip

\item[(2)] Let $u^0\in l^\infty(\Z)$ with $u^0_i\ge 0$ for $i\in\Z$ and $\gamma^{'}>0$ be given. If
$$
\liminf_{|i|\le \gamma^{'} t,t\to\infty} u_i(t;0,u^0)>0,
$$
then for any $0<\gamma<\gamma^{'}$,
\begin{equation}
\label{convergence-aux-eq1}
\limsup_{|i|\le \gamma t,t\to\infty} |u_i(t;0,u^0)-u^+(t)|=0.
\end{equation}
If
$$
\liminf_{s\in\R,|i|\le \gamma^{'} t, t\to\infty} u_i(t+s;s,u^0)>0,
$$
then for any $0<\gamma<\gamma^{'}$,
\begin{equation}
\label{convergence-aux-eq2}
\limsup_{|i|\le \gamma t,t\to\infty} |u_i(t+s;s,u^0)-u^+(t+s)|=0
\end{equation}
uniformly in $s\in\R$.
\end{itemize}
\end{theorem}

Thanks to Theorem \ref{stable-thm} (1),  $f$ satisfying (H1) and (H2) is also said to be of {\it monostable type}.

\smallskip

Next, we investigate spatial  spreading speeds of \eqref{main-eqn}.
Let $$l^\infty_0(\Z)=\{u=\{u_i\}_{i\in\Z}\in l^\infty(\Z)\,:\, u_i\ge 0\,\, \text{for all}\,\, i\in\Z,\,\, u_i=0\,\,\text{for}\,\, |i|\gg 1,\,\, \{u_i\}\not =0\}.
$$

\begin{definition}[Spreading speed interval]
(1) Let
$$c^*=\inf\{ c\,:\, \limsup_{|i|\ge ct, t\to\infty} u_i(t;0,u^0)=0\quad \text{for all}\,\, u^0\in l^\infty_0(\Z)\},$$
$$c_*=\sup\{c\,:\, \limsup_{|i|\le ct, t\to\infty} |u_i(t;0,u^0)-u^+(t)|=0\quad \text{for all}\,\, u^0\in l^\infty_0(\Z)\}.$$
We call $[c_*,c^*]$ the {\rm spreading speed interval} of \eqref{main-eqn}.

(2) Let
$$\tilde c^*=\inf\{ c\,:\, \limsup_{|i|\ge ct, t\to\infty} u_i(t+s;s,u^0)=0\,\, \text{uniformly in}\,\, s\in\R\,\, \text{for all}\,\, u^0\in l^\infty_0(\Z)\},$$
$$\tilde c_*=\sup\{c\,:\, \limsup_{|i|\le ct, t\to\infty} |u_i(t+s;s,u^0)-u^+(t+s)|=0\,\, \text{uniformly in}\,\,s\in\R\,\, \text{for all}\,\, u^0\in l^\infty_0(\Z)\}.$$
We call $[\tilde c_*,\tilde c^*]$ the {\rm generalized spreading speed interval} of \eqref{main-eqn}.
\end{definition}

Observe that
$$
\tilde  c_*\le  c_*\le  c^*\le \tilde  c^*,
$$
and that for any  $u^0\in l^\infty_0(\Z)$,
$$
 \limsup_{|i|\le ct, t\to\infty} |u_i(t;0,u^0)-u^+(t)|=0 \quad \forall\,\, c<c_*,
 $$
$$
 \limsup_{|i|\le ct, t\to\infty} |u_i(t+s;s,u^0)-u^+(t+s)|=0\quad \text{uniformly in }\,\, s\in\R\,\, \forall c<\tilde c_*,
 $$
 $$
 \limsup_{|i|\ge ct, t\to\infty} u_i(t;0,u^0)=0\quad \forall\,\, c>c^*,
 $$
 and
 $$
 \limsup_{|i|\ge ct, t\to\infty} u_i(t+s;s,u^0)=0\quad\text{uniformly in}\,\, s\in\R\,\, \forall\, c>\tilde c^*.
 $$
If $f(t,u)$ is independent of $t$ or periodic in $t$,  then
$$\tilde c_*=c_*=c^*=\tilde c^*
$$
(see \cite{LiZh1,LiZh2, We02}).

Define
$$
\bar f_{\inf}=\liminf_{t\ge s, t-s\to\infty}\frac{1}{t-s}\int_s^t f(\tau,0)d\tau,
$$
$$
\bar f_{\sup}=\limsup_{t\ge s, t-s\to\infty}\frac{1}{t-s}\int_s^t f(\tau,0)d\tau,
$$
$$
\bar f_{\inf}^+=\liminf_{t\ge s\ge 0, t-s\to\infty}\frac{1}{t-s}\int_s^t f(\tau,0)d\tau,
$$
and
$$
\bar f_{\sup}^+=\limsup_{t\ge s\ge 0, t-s\to\infty}\frac{1}{t-s}\int_s^t f(\tau,0)d\tau.
$$
We have the following theorem on the lower and upper bounds of the spreading speed intervals of \eqref{main-eqn}.

\begin{theorem}
\label{main-thm1}
Assume (H0)-(H2).
\begin{itemize}
\item[(1)]
$c_0^-:=\inf \limits_{\mu>0}\frac{e^{-\mu}+e^\mu-2+\bar f_{\inf}^+}{\mu}\le c_*\le
c^*\le c_0^+:=\inf \limits_{\mu>0}\frac{e^{-\mu}+e^{\mu}-2+\bar f_{\sup}^+}{\mu}.
$

\item[(2)] $\tilde c_0^-:=\inf \limits_{\mu>0}\frac{e^{-\mu}+e^\mu-2+\bar f_{\inf}}{\mu}\le \tilde c_*\le
\tilde c^*\le \tilde c_0^+:=\inf \limits_{\mu>0}\frac{e^{-\mu}+e^{\mu}-2+\bar f_{\sup}}{\mu}.
$
\end{itemize}
\end{theorem}

\begin{remark}
(1) If $f(\cdot,0)$ is unique ergodic, then the limit $$\hat f=\lim \limits_{t\ge s, t-s\to\infty}\frac{1}{t-s}\int_s^t f(\tau,0)d\tau$$ exists (see \cite{Sh09, Sh11} for details). Thus
$$c_*=c^*=\tilde c_*=\tilde c^*=\inf \limits_{\mu>0}\frac{e^{-\mu}+e^{\mu}-2+\hat f}{\mu}.$$
 In this case, $c_*$ is called the {\it spreading speed} of \eqref{main-eqn}.

 (2)  There is a unique $\mu^*>0$ such that
$$
\tilde c_0^-=\frac{e^{-\mu^*}+e^{\mu^*}-2+\bar f_{\inf}}{\mu^*}
$$
and
for any $\gamma>\tilde c_0^-$,
 the equation $\gamma=\frac{e^{-\mu}+e^{\mu}-2+\bar f_{\inf}}{\mu}$ has  exactly two positive solutions for $\mu$
 (see Lemma \ref{mu-star-lemma}).
 
  (3) It will be proved in Theorem \ref{main-thm2} that $\tilde c_*=\tilde c_0^-$.
\end{remark}

We then study transition front solutions of \eqref{main-eqn}.

\begin{definition} [Transition front] An entire solution $u(t)=\{u_{i}(t)\}_{i\in\Z}$ of \eqref{main-eqn} is called a  {\it transition front} (connecting $0$ and $u^+(t)$)  if $u_i(t)\in(0,u^+(t))$ for all $t\in\R$ and $i\in\Z$, and there exists
 $J:\R\to \Z$ such that
\begin{equation*}
\lim_{i\to-\infty}(u_{i+J(t)}(t)-u^+(t))=0\,\,\text{and}\,\,\lim_{i\to\infty}u_{i+J(t)}(t)=0\,\,\text{uniformly in}\,\,t\in\R.
\end{equation*}
%The functions $c$ and $\phi$ are respectively called the {\it speed} and {\it profile} of the transition front $u$.
\end{definition}

The notion of a transition front is a proper generalization of a traveling wave in homogeneous media or a periodic (or pulsating) traveling wave in periodic media. The {\it interface location function} $J(t)$ tells the position of the transition front $u(t)$ as time $t$ elapses. Notice, if $\xi(t)$ is a bounded integer-valued function, then $J(t)+\xi(t)$ is also an interface location function. Thus, interface location function is not unique. But, it is easy to check that if $\tilde J(t)$ is another interface location function, then $J(t)-\tilde J(t)$ is a bounded integer-valued function. Hence, interface location functions are unique up to addition by bounded integer-valued functions. The uniform-in-$t$ limits shows the \textit{bounded interface width}, that is,
\begin{equation*}
\forall\,\,0<\ep_{1}\leq\ep_{2}<1,\quad\sup_{t\in\R}{\rm diam}\{i\in\Z|\ep_{1}\leq u_{i}(t)\leq\ep_{2}\}<\infty.
\end{equation*}
We prove

\begin{theorem}
\label{main-thm2}
(1) For any $\gamma>\tilde c_0^-$,  let $0<\mu<\mu^*$ and  $c(t)=\frac{e^{-\mu}+e^{\mu}-2+f(t,0)}{\mu}$ be such that $\bar c_{\inf}=\gamma$. Then there exists a continuous function $\phi:\R\times\R\ra\R^+$ with $\phi(x,t)$ being non-increasing in $x$ and  \begin{equation*}
\lim_{x\to-\infty}(\phi(x,t)-u^+(t))=0\,\,\text{and}\,\,\lim_{x\to\infty}\frac{\phi(x,t)}{e^{-\mu x}}=1\,\,\text{uniformly in}\,\,t\in\R
\end{equation*}
 such that $u(t)$ is a transition front solution of \eqref{main-eqn}, where $u_i(t)=\phi(i-\int^t_0c(\tau)d\tau,t)$ for $i\in\Z$.
 
 \smallskip
 
  (2)  $\tilde c_*=\tilde c_0^-$.
\end{theorem}

\begin{remark}
(1) If $f(t,u)\equiv f(u)$ is independent of $t$, then so is $\phi(x,t)$ and hence
$u_i(t)=\phi(i- ct)$ is a traveling wave solution of \eqref{main-eqn} in the classical sense,
where $c=\frac{e^{-\mu}+e^{\mu}-2+f(0)}{\mu}$ (see Remarks \ref{super-solu-rk1} and \ref{super-solu-rk2}).

(2) If $f(t,u)$ is periodic in $t$ with period $T$, then so is $\phi(x,t)$ (see Remarks \ref{super-solu-rk1} and \ref{super-solu-rk2}). Let $\hat f=\frac{\int_0^T f(\tau,0)d\tau}{T}$ and
$\psi(x,t)=\phi(x-\frac{\int_0^t f(\tau,0)d\tau -\hat f t}{\mu},t)$. Then $\psi(x,t)$ is continuous in $x\in\R$ and $t\in\R$, nonincreasing in
$x$,  periodic in $t$ with period $T$,
and
$$
u_i(t)=\phi(i-\int_0^t c(\tau)d\tau,t)=\psi(i-ct,t),
$$
where $c=\frac{e^{-\mu}+e^{\mu}-2+\hat f}{\mu}$. Therefore, $u_i(t)=\psi(i-ct,t)$ is a periodic traveling wave solution
with $\psi: \R\times\R\to \R^+$ being continuous (see Remarks \ref{super-solu-rk1} and \ref{super-solu-rk2}), which is new. Observe that \cite[Theorem 4.2]{LiZh1} and \cite[Theorem 6.6]{We82} imply the existence of traveling wave
solution of \eqref{main-eqn} of the form  $u_i(t)=\Phi(i-ct,t)$, where for each fixed $t$, $\Phi(x,t)$ is only defined for $x\in\{i-c(nT+t)|n\in\Z\}$.
\end{remark}

{ We also prove

\begin{theorem}
\label{main-thm3}
There is a transition front solution $u_i^*(t)$ with interface location function $J^*(t)$ satisfying that $u_i^*(t)$ is nonincreasing in $i\in \Z$,
and
$$
\liminf_{t-s\to\infty}\frac{J^*(t)-J^*(s)}{t-s}=\tilde  c_0^-.
$$
\end{theorem}

\begin{remark}
The transition front solution in Theorem \ref{main-thm3} is the analogue of {\it critical traveling front solution} in literature
(see \cite{Na14}, \cite{Sh04}). It is also the analogue of the traveling wave solution with minimal wave speed in the time independent case.
\end{remark}
}

The rest of the paper is organized as follows. In Section 2, we establish some basic properties
of solutions of lattice equation \eqref{main-eqn} for the use in later sections. We study spatially homogeneous entire positive solutions of
\eqref{main-eqn} and prove Theorem \ref{stable-thm}  in Section 3. In Section 4, we investigate the (generalized) spreading speeds and prove
Theorem \ref{main-thm1}. Section 5 is devoted to the proof of the existence of transition fronts for lattice equation \eqref{main-eqn}.

%%%%%%%%%%%%%%%%%%%%%%%%%%%%%%%%%%%%%%%%

\section{Preliminary}

%\section{Basic properties of lattice KPP equations}

In this section, we present some preliminary materials to be used in later sections. We first present a comparison
principle for sub-solutions and super-solutions of \eqref{main-eqn} and prove the convergence of solutions on compact subsets. Next, we introduce the concept of the so called part metric and prove the decreasing property
of the part metric between two positive solutions of \eqref{main-eqn} as time increases. Finally, we present a technical lemma from \cite{NaRo12}.

First of all,  consider the following space continuous version of \eqref{main-eqn},
\begin{equation}\label{main-eqn2} \pa_tv(x,t)=Hv(x,t)+v(x,t) f(t,v(x,t)), \quad\quad x\in \R,\, t\in\R,
\end{equation}
where  $$Hv(x,t)=v(x+1,t)+v(x-1,t)-2v(x,t), \quad x\in \R,\, t\in\R.$$
Recall
$$
l^\infty(\Z)=\{u:\Z\to\R\,:\, \sup_{x\in\Z}|u(x)|<\infty\}.
$$
Let
$$
l^\infty(\R)=\{u:\R\to\R\,:\, \sup_{x\in\R}|u(x)|<\infty\}
$$
with norm $\|u\|=\sup_{x\in\R}|u(x)|$.
Let  $$ l^{\infty,+}(\Z)=\{u\in l^{\infty}(\Z):\inf_{i\in\Z}u_i\ge 0\},\quad l^{\infty,+}(\R)=\{u\in l^\infty(\R)\,:\, \inf_{x\in\R}u(x)\ge 0\}$$
and
$$ l^{\infty,++}(\Z)=\{u\in l^{\infty}(\Z):\inf_{i\in\Z}u_i> 0\},\quad l^{\infty,++}(\R)=\{u\in l^{\infty}(\R): \inf_{x\in\R} u(x)>0\}.$$
For $u,v\in l^\infty(\R)$ (resp. $u,v\in l^\infty(\Z)$), we define
$$
u\ge v \quad {\rm if}\quad u-v\in l^{\infty,+}(\R)\quad {\rm (resp.}\,\,\, u-v\in l^{\infty,+}(\Z)),
$$
and
$$
u\gg v \quad {\rm if}\quad u-v\in l^{\infty,++}(\R)\quad {\rm (resp.}\,\,\, u-v\in l^{\infty,++}(\Z)).
$$

For any $u_0\in l^\infty(\R)$, let $u(x,t;s,u_0)$ be the solution of \eqref{main-eqn2} with $u(x,s;s,u_0)=u_0(x)$, and
for any $u^0\in l^\infty(\Z)$, let $u(t;s,u^0)=\{u_i(t;s,u^0)\}_{i\in\Z}$ be the solution of \eqref{main-eqn} with $u_i(s;s,u^0)=u^0_i$ for $i\in\Z$.
Observe that for given $u_0\in l^\infty(\R)$ and $x_0\in\R$, $u(x_0+i,t;s,u_0)$ only depends on $\{u_0(x_0+i)|i\in\Z\}$ and
$u(x_0+i,t;s,u_0)=u_i(t;s,u^0)$, where $u^0_i=u_0(x_0+i)$ for $i\in\Z$.

 A  function $v(x,t)$ on $\R\times[s,T)$ which is continuous in $t$
  is called a {\it super-solution} or {\it sub-solution} of \eqref{main-eqn2} (resp. \eqref{main-eqn}) if for any given $x\in\R$ (resp. $x\in\Z$), $v(x,t)$ is  absolutely continuous
in  $t\in [s,T)$,  and
$$ v_t(x,t)\ge Hv(x,t)+v(x,t) f(t,v(x,t))\quad {\rm for}\quad a.e. \,\, t\in [s,T)$$
or
$$ v_t(x,t)\le  Hv(x,t)+v(x,t) f(t,v(x,t))\quad {\rm for}\quad a.e. \,\, t\in[s,T).$$

\begin{proposition}[Comparison principle]\label{comparison}
\begin{itemize}
\item[(1)]
If $u_1(x,t)$ and $u_2(x,t)$ are bounded sub-solution and super-solution of \eqref{main-eqn2} (resp. \eqref{main-eqn}) on $[s,T)$, respectively, and $u_1(\cdot,0)\leq u_2(\cdot,0)$, then $u_1(\cdot,t)\leq u_2(\cdot,t)$ for $t\in[s,T)$.

\item[(2)] Suppose that $u_1(x,t)$, $u_2(x,t)$ are bounded and satisfy that for any given $x\in\R$ (resp. $x\in\Z$),
 $u_1(x,t)$ and $u_2(x,t)$ are absolutely continuous in $t\in[s,\infty)$, and

\medskip
 \item[]  $\pa_t u_2(x,t)-(Hu_2(x,t)+u_2(x,t)f(t,u_2(x,t)))>\pa_t u_1(x,t)-(Hu_1(x,t)+u_1(x,t)f(t,u_1(x,t)))$

\medskip
    \item[] for a.e. $t>s$.
    Moreover, suppose that $u_2(\cdot,s)\geq u_1(\cdot,s)$. Then $u_2(x,t)>u_1(x,t)$ for $x\in\R$ (resp. $x\in\Z$), $t>s$.

\item[(3)] If $u_0\in l^{\infty,+}(\R)$ (resp. $u^0\in l^{\infty,+}(\Z)$, then $u(x,t;s,u_0)$ (resp. $u(t;s,u^0)$) exists and
$u(\cdot,t;s,u_0)\ge 0$ (resp. $u(t;s,u^0)\ge 0$) for all $t\ge s$.
\end{itemize}
\end{proposition}

\begin{proof} We prove the proposition for \eqref{main-eqn2}. It can be proved similarly for \eqref{main-eqn}.

\smallskip

(1) We prove (1)  by modifying the arguments of \cite[Proposition 2.4]{HuShVi08}.

Let $w(x,t)=e^{ct}(u_2(x,t)-u_1(x,t))$, where $c$ is a constant to be determined later. Then for any given $ x\in\R$,
there is a measurable subset $E$ of $[s,T]$ with Lebesgue measure $0$ such that
\begin{align}\label{difference}
\pa_tw(x,t)&\geq Hw(x,t)+(a(x,t)+c)w(x,t)\nonumber\\&=w(x+1,t)+w(x-1,t)+(a(x,t)-2+c)w(x,t)
\end{align}
for  $t\in[s,T]\setminus E$, where
$$a(x,t)=f(t,u_2(x,t))+u_1(x,t)\int^1_0f_u(t,su_1(x,t)+(1-s)u_2(x,t))ds\quad\mbox{ for }x\in\R,\,t\in[s,T].$$
Let $p(x,t)=a(x,t)-2+c$. By the boundedness of $u_1$ and $u_2$, we can choose $c>0$ such that
$$\inf _{(x,t)\in\R\times [s,T]}p(x,t)>0.$$
We claim that  $w(x,t)\geq 0$ for $x\in\R$ and $t\in[s,T]$.

Let $p_0=\sup_{(x,t)\in\R\times [s,T]}p(x,t)$. It suffices to prove the claim for $x\in\R$ and $t\in(s,T_0]$ with $T_0=s+\min(T-s,\frac{1}{p_0+2})$. Assume that there are $\tilde x\in\R$ and $\tilde t\in(s,T_0]$ such that $w(\tilde x,\tilde t)<0$. Then there is $t^0\in(s,T_0)$ such that
$$w_{\inf}:=\inf_{(x,t)\in\R\times [s,t^0]}w(x,t)<0.$$ Observe that there are $x_n\in\R$ and $t_n\in(s,t^0]$ such that
$$w(x_n,t_n)\to w_{\inf}\quad\mbox{ as }\,n\to\infty.$$
By \eqref{difference}  and the fundamental theorem of calculus for Lebesgue integrals, we get
\begin{align}
w(x_n,t_n)-w(x_n,s)&\geq\int^{t_n}_s[w(x_n+1,t)+w(x_n-1,t)+p(x_n,t)w(x_n,t)]dt \nonumber\\&\geq\int^{t_n}_s[2w_{\inf}+p(x_n,t)w_{\inf}]dt \nonumber\\&\geq (t^0-s)(2+p_0)w_{\inf}\quad\quad\mbox{ for }\,n\geq1.\nonumber
\end{align}
Note that $w(x_n,s)\geq0$, we then have
$$w(x_n,t_n)\geq (t^0-s)(2+p_0)w_{\inf}\quad\quad\mbox{ for }\,n\geq1.$$
Letting $n\to\infty$, we obtain
$$w_{\inf}\geq (t^0-s)(2+p_0)w_{\inf}>w_{\inf}.$$
A contradiction. Hence the claim is true and $u_1(x,t)\leq u_2(x,t)$ for $x\in\R$ and $t\in[s,T]$.

(2) By the similar arguments as getting \eqref{difference}, we can find $c$, $\mu>0$ such that
for any given $x\in\R$,
$$\pa_t w(x,t)>w(x+1,t)+w(x-1,t)+\mu w(x,t)\quad\mbox{ for }\, a.e.\, \, t>s,$$
where $w(x,t)=e^{ct}(u_2(x,t)-u_1(x,t))$. Then we have that for any given $x\in\R$,
$$
w(x,t)>w(x,s)+\int_s ^t \big(w(x+1,\tau)+w(x-1,\tau)+\mu w(x,\tau)\big)d\tau.
$$
By the arguments in (1), $w(x,t)\ge 0$ for all $x\in\R$ and $t\ge s$. It then follows that
$w(x,t)>w(x,s)\ge 0$ and hence $u_2(x,t)>u_1(x,t)$ for all $x\in\R$ and $t>s$.

(3) By (1) and (H1), for any $u_0\in l^{\infty,+}(\R)$, $0\le u(\cdot,t;s,u_0)\le \max\{\|u_0\|,M_0\}$ for all $t>s$ in the existence interval of $u(\cdot,t;s,u_0)$. It then follows that $u(\cdot,t;s,u_0)$ exists and $u(\cdot,t;s,u_0)\ge 0$ for all $t\ge s$.
\end{proof}

\begin{proposition}
\label{convergence-prop}
Suppose that $u_{0n},u_0\in l^{\infty,+}(\R)$ $(n=1,2,\cdots)$ with $\{\|u_{0n}\|\}$ being  bounded.
If $u_{0n}(x)\to u_0(x)$ as $n\to\infty$
 uniformly in $x$ on bounded sets,
{ then for each $t>0$}, $u(x,s+t;s,u_{0n})- u(x,s+t;s,u_0)\to 0$ as $n\to\infty$ uniformly in $x$ on bounded sets and $s\in\R$.
\end{proposition}

\begin{proof}
It can be proved by the similar arguments in \cite[Proposition 3.3]{KoSh}. For the completeness, we provide a proof in the following.

 Let $v^n(x,t;s)=u(x,s+t;s,u_{0n})-u(x,s+t;s,u_0)$.
Then $v^n(t,x;s)$ satisfies
\begin{equation*}
v^n_t(x,t;s)=H v^n(x,t;s)+a_n(t,x;s)v^n(x,t;s),
\end{equation*}
where
\begin{align*}
a_n(t,x;s)=&f(s+t,u(x,s+t;s,u_{0n}))\\
&+u(x,s+t;s,u_{0})\cdot \int_0^1  f_u(s+t,r u(x,s+t;s,u_{0n})+
(1-r)u(x,s+t;s,u_0))dr.
\end{align*}
Observe that $\{a_n(t,x;s)\}$ is uniformly bounded.

Take a $\lambda>0$. Let
$$
X(\lambda)=\{u:\R\to\R\,|\, u(\cdot) e^{-\lambda |\cdot|}\in l^\infty(\R)\}
$$
with norm $\|u\|_\lambda=\|u(\cdot)e^{-\lambda |\cdot|}\|_{l^\infty(\R)}$.
Note that $H:X(\lambda)\to X(\lambda)$ generates an analytic  semigroup,
and there are $M>0$ and $\omega>0$ such that
$$
\|e^{H t}\|_{X(\lambda)}\leq M e^{\omega t}\quad \forall { t\ge 0}.
$$
Hence
\begin{align*}
v^n(\cdot,t;s)=&e^{ Ht}v^n(\cdot,0;s)+\int_0^t e^{ H(t-\tau)}a_n(\tau,\cdot;s)v^n(\cdot,\tau;s)
d\tau
\end{align*}
and
then
\begin{align*}
\|v^n(\cdot,t;s)\|_{X(\lambda)}&\leq M e^{\omega t}\|v^n(\cdot,0;s)\|_{X(\lambda)}+
M\sup_{s\in\R,\tau\in[0,t],x\in\R} |a_n(\tau,x;s)|\int_0^ t e^{\omega(t-\tau)}\|v^n(\cdot,\tau;s)\|_{X(\lambda)}d\tau.
\end{align*}
By Gronwall's inequality,
$$
\|v^n(\cdot,t;s)\|_{X(\lambda)}\leq e^{(\omega+M\sup_{s\in\R,\tau\in[0,t],x\in\R} |a_n(\tau,x;s)|)t}\big(M\|v^n(\cdot,0;s)\|_{X(\lambda)}\big).
$$
Note that $\|v^n(\cdot,0;s)\|_{X(\lambda)}\to 0$ uniformly in $s\in\R$. It then
follows that
$$
\|v^n(\cdot,t;s)\|_{X(\lambda)}\to 0\quad {\rm as}\quad n\to\infty
$$
uniformly in $s\in\R$ and then
$$
u(x,s+t;s,u_{0n})-u(x,s+t;s,u_0)\to 0 \quad {\rm as}\quad
n\to\infty
$$
uniformly in $x$ on bounded sets and $s\in\R$.
\end{proof}

Next, we introduce the so called part metric and prove the decreasing property of the part metric between two
positive solutions as time increases.
For given $u,v\in l^{\infty, +}(\Z)$ (resp. $u,v\in l^{\infty,+}(\R)$),  if
$$\{\alpha>1\, :\, \frac{1}{\alpha}v\leq u\leq \alpha v  \big\}\not =\emptyset,
$$
we define $\rho(u,v)$ by
$$\rho(u,v):=\inf\big\{\ln\alpha:\alpha>1,\frac{1}{\alpha}v\leq u\leq \alpha v  \big\}$$
and call
$\rho(u,v)$ the {\it part metric between $u$ and $v$}.

 Observe that if $u,v\in l^{\infty,++}(\Z)$ (resp. $u,v\in l^{\infty,++}(\R)$), then $\rho(u,v)$ is well defined.
Observe also that if $u,v\in l^{\infty,+}(\Z)$ (resp. $u,v\in l^{\infty,+}(\R)$), and $u_i>0,v_i>0$ for all $i\in\Z$ (resp. $u(x)>0,v(x)>0$ for all $x\in\R$), $\inf_{i\le i_0}u_i>0$, $\inf_{i\le i_0}v_i>0$ for any $i_0\in \Z$ (resp. $\inf_{x\le x_0}u(x)>0$, $\inf_{x\le x_0}v(x)>0$ for any $x_0\in\R$), and
$\lim_{i\to\infty}\frac{u_i}{e^{-\mu i}}=\lim_{i\to\infty}\frac{v_i}{e^{-\mu i}}=1$ (resp. $\lim_{x\to\infty}\frac{u(x)}{e^{-\mu x}}=1$,
 $\lim_{x\to\infty} \frac{v(x)}{e^{-\mu x}}=1$) for some $\mu>0$, then $\rho(u,v)$ is also  well defined.

\begin{proposition}[Part metric]
\label{part-metric}
\begin{itemize}
\item[(1)] For given  $u_0,v_0\in l^{\infty,+}(\R)$ with $u_0\not =v_0$, if $\rho(u_0,v_0)$ is well defined, then  $\rho(u(\cdot,t;s,u_0),u(\cdot,t;s,v_0))$ is also well defined for
every $t>s$ and $\rho(u(\cdot,t;s,u_0),u(\cdot,t;s,v_0))$  decreases as $t$ increases.

\item[(2)] For any $\epsilon>0$, $\sigma>0$, $M>0$,  and $\tau>0$  with $\epsilon<M$ and
$\sigma\le \ln \frac{M}{\epsilon}$, there is $\delta>0$   such that
for any $u_0,v_0\in l^{\infty,++}(\R)$ with $\epsilon\le u_0(x)\le M$, $\epsilon\le v_0(x)\le M$ for $x\in\R$ and
$\rho(u_0,v_0)\ge\sigma$, there holds
$$
\rho(u(\cdot,\tau+s;s,u_0),u(\cdot,\tau+s;s,v_0))\le \rho(u_0,v_0)- \delta\quad\forall\,\, s\in\R.
$$

\item[(3)] Suppose that $u_1(x,t)$ and $u_2(x,t)$ are two distinct positive entire solutions of \eqref{main-eqn2} and
that  there are $c(t)$ and $\mu>0$ such that
$$
\lim_{x\to\infty}\frac{u_i(x+c(t),t)}{e^{-\mu x}}=1
$$
uniformly in $t$ ($i=1,2$) and for any $x_0\in \R$,
$$
\inf_{x\le x_0,t\in\R}u_i(x+c(t),t)>0
$$
for $i=1,2$. Then for any $\tau>0$ and $T\in\R$, there is $\delta>0$ such that
$$
\rho(u_1(\cdot,s+\tau),u_2(\cdot,s+\tau))<\rho(u_1(\cdot,s),u_2(\cdot,s))-\delta
$$
for $s\le T$.
\end{itemize}
\end{proposition}

\begin{proof}
(1) Suppose that $u_0,v_0\in l^{\infty,+}(\R)$ are such that $u_0\not =v_0$ and $\rho(u_0,v_0)$ is well defined. Then there is
 $\alpha>1$ such that $\rho(u_0,v_0)=\ln\alpha$ and
 $$
 \frac{1}{\alpha} v_0\le u_0\le \alpha v_0.
 $$
 By Proposition \ref{comparison}  and $f_u(t,u)<0$ for $u\geq0$ (which implies that $f(t,\alpha u)\le f(t,u)\le f(t,\frac{1}{\alpha}u)$), we have
$$
\frac{1}{\alpha} u(x,t;s,v_0)<u(x,t;s,\frac{1}{\alpha}v_0)\le u(x,t;s,u_0)\le u(x,t;s,\alpha v_0)<\alpha u(x,t;s,v_0)
$$
for all $x\in\R$, $t>s$. It then follows that
$$
\rho(u(\cdot,t;s,u_0),u(\cdot,t;s,v_0))<\rho(u_0,v_0)
$$
for any $t>s$ and then for any $t_2>t_1\ge s$,
\begin{align*}
\rho(u(\cdot,t_2;s,u_0),u(\cdot,t_2;s,v_0))&=\rho(u(\cdot,t_2;t_1,u(\cdot,t_1;s,u_0)),u(\cdot,t_2;t_1,u(\cdot,t_1;s,v_0)))\\
&<\rho(u(\cdot,t_1;s,u_0),u(\cdot,t_1;s,v_0)).
\end{align*}
(1) is thus proved.

(2) It can be proved by the similar arguments as in \cite[Proposition 3.4]{KoSh}. For the self-completeness, we provide a proof
in the following.

Let $\epsilon>0$, $\sigma>0$, $M>0$, and $\tau>0$ be given and $\epsilon<M$, $\sigma<\ln \frac{M}{\epsilon}$. First, note that by Proposition \ref{comparison},
 there are $\epsilon_1>0$ and $M_1>0$ such that
for any $u_0\in l^{\infty,++}(\R)$ with $\epsilon\le u_0(x)\le M$ for $x\in\R$, there holds
\begin{equation}
\label{part-metric-eq1}
\epsilon_1\le u(\cdot,t+s;s,u_0)\le M_1\quad \forall\,\, t\in[0,\tau],\,\, s\in\R.
\end{equation}
Let
\begin{equation}
\label{part-metric-eq2}
\delta_1=\epsilon_1^2 e^\sigma (1-e^\sigma)\sup_{t\in\R,u\in[\epsilon_1,M_1M/\epsilon]}f_u(t,u).
\end{equation}
Then $\delta_1>0$ and there is $0<\tau_1\le\tau$ such that
\begin{equation}
\label{part-metric-eq3-1}
\frac{\delta_1}{2}\tau_1<e^\sigma\epsilon_1
\end{equation}
and
\begin{equation}
\label{part-metric-eq3-2}
\Big|\frac{\delta_1}{2}tvf_u(t+s,w)\Big|+\Big|\frac{\delta_1}{2}tf(t+s,v-\frac{\delta_1}{2}t)\Big|\le\frac{\delta_1}{2}\quad \forall \,\,s\in\R,\,\, t\in [0,\tau_1],\,\, v,w\in[0,M_1M/\epsilon].
\end{equation}
Let
\begin{equation}
\label{part-metric-eq4}
\delta_2=\frac{\delta_1\tau_1}{2 M_1}.
\end{equation}
Then $\delta_2<e^\sigma$ and $0<\frac{\delta_2 \epsilon}{M}<1$. Let
\begin{equation}
\label{part-metric-eq5}
\delta=-\ln\big(1-\frac{\delta_2\epsilon}{M}\big).
\end{equation}
Then $\delta>0$. We prove that  $\delta$ defined in \eqref{part-metric-eq5}  satisfies the property in the proposition.

For any $u_0,v_0\in l^{\infty,++}(\R)$ with $\epsilon\le u_0(x)\le M$ and $\epsilon\le v_0(x)\le M$ for $x\in\R$ and
$\rho(u_0,v_0)\ge\sigma$,  there is $\alpha^*> 1$ such that  $\rho(u_0,v_0)=\ln \alpha^{*}$
and $\frac{1}{\alpha^{*}}u_0\leq v_0\leq \alpha^{*} u_0$. Note that $e^\sigma\le\alpha^*\le \frac{M}{\epsilon}$.
By (1),  $\rho(u(\cdot,t;s,u_0),u(\cdot,t;s,v_0))$ is non-increasing in $t>s$.
We prove that
$$
\rho(u(\cdot,s+\tau;s,u_0),u(\cdot,s+\tau;s,v_0))\le\rho(u_0,v_0)-\delta\quad\forall\,\, s\in\R.
$$

Let
 $$v(x,t)=\alpha^{*}u( x,t; s, u_0).$$
Note that $e^\sigma\le \alpha^*\le \frac{M}{\epsilon}$
and
\begin{align*}
v_{t}(x,t)&=H v(x,t)+v(x,t)f(t,u(x,t;s,u_0))\\
& =Hv(x,t)+v(x,t)f(t, v(x,t))+v(x,t)f(t, u(x,t;s,u_0))-v(x,t)f(t,v(x,t))\\
&\ge H v (x,t)+ v(x,t)f(t, v(x,t))+\delta_1\quad \forall\,  s<t\le s+\tau_1,\,\, s\in\R.
\end{align*}  This together with \eqref{part-metric-eq3-1}, \eqref{part-metric-eq3-2}
implies that 
$$
(v(x,t)-\frac{\delta_1}{2}(t-s))_t\ge H\big(v(x,t)-\frac{\delta_1}{2}(t-s)\big)+\big(v(x,t)-\frac{\delta_1}{2}(t-s)\big)f\big(t,v(x,t)-\frac{\delta_1}{2}(t-s)\big)
$$
for $s<t\le s+ \tau_1$.
Then by Proposition \ref{comparison} again,
$$
u(\cdot,t;s,\alpha^* u_0)\leq \alpha^* u(\cdot,t;s,u_0)-\frac{\delta_1}{2}(t-s)\quad {\rm for}\quad s<t\le s+\tau_1.
$$
By \eqref{part-metric-eq4},
$$
u(\cdot,s+\tau_1;s,v_0)\le (\alpha^*-\delta_2) u(\cdot,s+\tau_1;s,u_0).
$$
Similarly, it can be proved that
$$
\frac{1}{\alpha^*-\delta_2}u(\cdot,s+\tau_1;s,u_0)\le u(\cdot,s+\tau_1;s,v_0).
$$
It then follows that
$$
\rho(u(\cdot,s+\tau_1;s,u_0),u(\cdot,s+\tau_1;s,v_0))\le \ln(\alpha^*-\delta_2) =\ln\alpha^*+\ln (1-\frac{\delta_2}{\alpha^*})\le \rho(u_0,v_0)-\delta.
$$
and hence
$$
\rho(u(\cdot,s+\tau;s,u_0),u(\cdot,s+\tau;s,v_0))\le\rho(u(\cdot,s+\tau_1;s,u_0),u(\cdot,s+\tau_1;s,v_0))\le\rho(u_0,v_0)-\delta.
$$

(3) Without loss of generality, we assume that $T=0$ and  fix any $\tau>0$.  Let $\rho(t)=\rho(u_1(\cdot,t),u_2(\cdot,t))$. Then
there is $\alpha(t)>1$ such that $\rho(t)=\ln\alpha(t)$. We have
$$
\frac{1}{\alpha(0)}u_2(x,0)\le u_1(x,0)\le \alpha(0) u_2(x,0),\,\, \forall\,\, x\in\R.
$$
Note  that
\begin{equation}
  \label{limit-eq1-1}
  \lim_{x\to\infty}
\frac{u_i(x+c(t),t)} { e^{-\mu
x}}=1,
\end{equation}
 uniformly in $t$. This implies that for
 any $\epsilon>0$ with $\frac{1+\epsilon}{1-\epsilon}<\alpha(0)(\le \alpha(t)$ for $t\le 0)$,
there is $M_\epsilon>0$ such that
\begin{equation}
\label{part-metric-aux-eq1}
\frac{1-\epsilon}{1+\epsilon} \, u_2(x+c(t),t)\le
u_1(x+c(t),t)\le\frac{1+\epsilon}{1-\epsilon}\,  u_2(x+c(t),t)
\end{equation}
for $x\ge M_\epsilon$ and all $t$. Note also
that there is $\sigma_\epsilon>0$ such that for $x\le
M_\epsilon$ and all $t$, there holds,
\begin{equation}
\label{part-metric-aux-eq2}
u_i(x+c(t),t)\ge \sigma_\epsilon.
\end{equation}

For any given $s\le 0$, let
 $\tilde u(x,t)=\alpha(s) u_2(x,t)$. By \eqref{part-metric-aux-eq2},  there is $\delta_\epsilon>0$ such that
\begin{align}
\label{part-metric-aux-eq3}
\tilde u_t(x,t)&=H \tilde u(x,t)+\tilde
u(x,t)f(t, u_2(x,t))\nonumber\\
&\ge H \tilde u(x,t)+\tilde
u(x,t)f(t, \tilde u(x,t))+\delta_\epsilon
\end{align}
for $x\le M_\epsilon +c(t)$, $s\le t\le s+\tau$, and $s\le 0$.
Let $\hat u(x,t)=u(x,t;s,\alpha(s)u_2(\cdot,s))$. Note that
$$
\hat u_t(x,t)=H\hat u(x,t)+\hat u(x,t)f(t,\hat u(x,t))
$$
for all $x\in\R$ and  $\tilde u(x,t)> \hat u(x,t)$ for $x\in\R$ and $t\ge s$.
Let $w(x,t)=\tilde u(x,t)-\hat u(x,t)$. Then
\begin{align*}
w_t(x,t)&\ge w(x+1,t)+w(x-1,t)-2 w(x,t)+\tilde u(x,t)f(t,\tilde u(x,t))-\hat u(x,t)f(t,\hat u(x,t))+\delta_\epsilon\\
&\ge p(x,t)w(x,t)+\delta_\epsilon
\end{align*}
for $x\le M_\epsilon +c(t)$, where
$$
p(x,t)=-2+\Big[\tilde u(x,t)f(t,\tilde u(x,t))-\hat u(x,t)f(t,\hat u(x,t))\Big]/[\tilde u(x,t)-\hat u(x,t)].
$$
It then follows that
$$
w(x,t)\ge \int_s^t e^{\int_r^t (-2+p(x,\tau))d\tau}\delta_\epsilon dr
$$
for $x\le M_\epsilon+c(t)$.
This implies that there is $\tilde \delta_\epsilon>0$ such that
\begin{equation*}
\tilde u(x,s+\tau)\ge \hat u(x,s+\tau)+\tilde \delta_\epsilon
\end{equation*}
for $x\le M_\epsilon+c(s+\tau)$. It follows that
\begin{equation}
\label{part-metric-aux-eq4}
u_1(x,s+\tau)\le  \hat u(x,s+\tau)\le \alpha(s) u_2(x,s+\tau)-\tilde \delta_\epsilon
\end{equation}
for $x\le M_\epsilon+c(s+\tau)$.

By \eqref{part-metric-aux-eq1} and \eqref{part-metric-aux-eq4},  there is  $0<\delta<\alpha(0)(<\alpha(s))$ such that
$$
u_1(x,s+\tau)\le (\alpha(s)-\delta) u_2(x,s+\tau)$$
for $x\in\R$ and $s\le 0$. Similarly, we can prove that
$$
u_1(x,s+\tau)\ge \frac{1}{\alpha(s)-\delta}u_2(x,s+\tau)
$$
for all $x\in\R$ and $s\le 0$. (3) is thus proved.
\end{proof}

Finally, we present a technical lemma from  \cite{NaRo12}.
Let
$$
\bar f_T=\inf_{k\in \N}\frac{1}{T}\int_{(k-1)T}^{kT} f(\tau,0)d\tau.
$$

 \begin{lemma}
\label{technical-lemma}
\begin{itemize}
\item[(1)] Let $B\in L^{\infty}(\R)$. Then
$$\bar B_{\inf}=\sup_{A\in W^{1,\infty}(\R)} \essinf_{t\in\R}(A^{\prime}+B)(t).$$

\item[(2)]
For given $T>0$,
there is $A\in W^{1,\infty}((0,\infty))$ such that
$$
\essinf_{t\in (0,\infty)}\big(A^{'}(t)+f(t,0)\big)=\bar f_T.
$$

%$$
%\bar f_T=\inf_{k\in \N}\frac{1}{T}\int_{(k-1)T}^{kT} f(\tau,0)d\tau.
%$$

\item[(3)] $$\bar{f}^+_{\inf}=\lim_{T\to\infty}\inf_{t\geq 0}\frac{1}{T}\int^{t+T}_tf(\tau,0)d\tau
$$
and
$$\bar{f}_{\inf}=\lim_{T\to\infty}\inf_{t\in\R}\frac{1}{T}\int^{t+T}_tf(\tau,0)d\tau
$$

\end{itemize}
\end{lemma}

\begin{proof}
(1) It follows  from \cite[Lemma 3.2]{NaRo12}.

(2) It follows from  \cite[Lemma 3.2, Remark 3.3]{NaRo12}.

(3) It follows from \cite[Proposition 3.1]{NaRo12}.
\end{proof}

\section{Entire positive solutions}

In this section, we study entire positive solutions of \eqref{main-eqn} and prove Theorem \ref{stable-thm}.

\begin{proof}[Proof of Theorem \ref{stable-thm}]
(1)
First, we consider
\begin{equation}\label{ode}
\dot u=uf(t,u),\quad\quad t\in\R.
\end{equation}
For any $u_0\in\R$, let $u(t;s,u_0)$ be the solution of \eqref{ode} with $u(s;s,u_0)=u_0$. We prove that \eqref{ode} has
an entire   solution $u^+(t)$ with $\inf_{t\in\R}u^+(t)>0$.

Consider the  linearization of \eqref{ode} at 0,
\begin{equation}\label{ode linearization}
\dot v=f(t,0)v,\quad\quad t\in\R.
\end{equation} Let $v(t;s,v_0)$ be the solution of \eqref{ode linearization} with $v(s;s,v_0)=v_0$. Then
$$v(t;s,v_0)=e^{\int^t_sf(\tau,0)d\tau}v_0.$$
By (H1) we can find $\epsilon_0>0$ and $T>0$ such that
$$
\frac{\int^{s+T}_sf(\tau,0)d\tau}{T}>\epsilon_0\quad \forall\,\, s\in\R.
$$
Note that for the above  $\epsilon_0>0$, there is $\delta_0>0$ such that
$$f(t,u)\geq f(t,0)-\epsilon_0\quad\mbox{ for all }\,t\in\R\,,|u|\leq \delta_0.$$
 Let $v_0>0$ be such that $e^{\int^t_sf(\tau,0)d\tau}v_0\leq\delta_0$ for all $s\in\R$ and $t\in[s,s+T]$. Then by the comparison principle
 for scalar ODEs,
$$u(t;s,v_0)\geq e^{\int^t_sf(\tau,0)d\tau-\epsilon_0(t-s)}v_0\quad\mbox{ for }\,s\in\R,\,t\in[s,s+T].
$$
In particular,
$$
u(s+T;s,v_0)\geq e^{\int^{s+T}_sf(\tau,0)d\tau-\epsilon_0T}v_0\ge v_0.
$$
By induction, we have
\begin{equation}
\label{positive-solu-eq1}
u(t;s,v_0)\geq e^{\int^t_{s+nT}f(\tau,0)d\tau-\epsilon_0(t-s-nT)}v_0\quad\mbox{ for }\,s\in\R,\,t\in[s+nT,s+(n+1)T],
\end{equation}
where $n=0,1,2,\cdots$.
By (H1), $f(t,u)<0$ for all $t\in\R$ and $u\ge M_0$. Then
\begin{equation}
\label{positive-solu-eq2}
u(t;s,M_0)<M_0\quad\mbox{ for }\,t>s.
\end{equation}

Let
$$u^n(t)=u(t;-nT,M_0),\quad t\ge -nT.$$
Then we get
$$u(t;-(n+1)T,v_0)<u^{n+1}(t)<u^n(t),\quad t\ge -nT.$$
Let $$u^+(t)=\lim_{n\to\infty}u^n(t).$$ We have that $u^+(t)$ is an entire  solution of
 \eqref{ode} and then that $\{u^+_i(t)=u^+(t)\}_{i\in\Z}$ is a spatially homogeneous solution of \eqref{main-eqn}.
 By \eqref{positive-solu-eq1},
 \begin{equation}
 \label{positive-solu-eq3}
 \inf \limits_{t\in\R}u^+(t)>0.
\end{equation}
   Hence $\{u^+_i(t)=u^+(t)\}_{i\in\Z}$ is a spatially homogeneous entire positive solution of \eqref{main-eqn}.
If no confusion occurs, we may still  write  $\{u^+_i(t)=u^+(t)\}_{i\in\Z}$ as $u^+(t)$.

Next, we claim  that for any $u^0\in l^{\infty,++}(\Z)$,
$$\|u(t+s;s,u^0)-u^+(t+s)\|_\infty \to 0\quad {\rm as}\quad t\to\infty$$
uniformly in $s\in\R$.
Assume that there is $u^0\in l^{\infty,++}(\Z)$ such that $\|u(t+s;s,u^0)-u^+(t+s)\|_\infty$ does not converge to $0$ as $t\to\infty$
uniformly in $s\in\R$. Then there are $\tilde \epsilon_{0}>0$, $s_n\in \R$, and $t_n\in\R$ with $t_n\to\infty$ as $n\to\infty$ such that
\begin{equation}
\label{positive-solu-eq4}
\|u(t_n+s_n;s_n,u^0)-u^+(t_n+s_n)\|_\infty\ge \tilde \epsilon_{0}\quad \forall\,\, n\ge 1.
\end{equation}
By Proposition \ref{part-metric}(1),
$$
\rho(u(t+s_n;s_n,u^0),u^+(t+s_n))<\rho(u^0,u^+(s_n))\quad \forall\,\, t>0.
$$
This together with \eqref{positive-solu-eq3} implies that there are $0<\epsilon<M$ such that
\begin{equation}
\label{positive-solu-eq5}
\epsilon\le u(t+s_n;s_n,u^0)\le M,\quad \epsilon\le u^+(t+s_n)\le M\quad \forall \,\, t\ge s_n,\,\, n=1,2,\cdots.
\end{equation}
By \eqref{positive-solu-eq4}, \eqref{positive-solu-eq5},  and Proposition \ref{part-metric}(2),
 there are $\tilde \sigma_0>0$, $\tilde\delta_0>0$, and $\tau>0$ such that
\begin{align*}
\tilde \sigma_0&\le\rho(u(t_n+s_n;s_n,u^0),u^+(t_n+s_n))\\
 &\le \rho(u(k\tau+s_n;s_n,u^0),u^+(k\tau+s_n))\\
 &\le \rho(u^0,u^+(s_n))-k\tilde\delta_0\quad \forall\,\, n\ge 1,\,\, 1\le k\le [t_n/\tau].
\end{align*}
This is a contradiction. Hence the claim holds.

By the claim in the above, \eqref{main-eqn} has only one spatially homogeneous entire positive solution. (1) is thus proved.

(2) We prove \eqref{convergence-aux-eq2}. \eqref{convergence-aux-eq1} can be proved similarly.

Let
$$
\delta_0=\liminf_{s\in\R,|i|\le\gamma^{'} t, t\to\infty} u_i(s+t;s,u^0).
$$
Then there is $T>0$ such that
$$
u_i(s+t;s,u^0)\ge \frac{\delta_0}{2}\quad \forall\,\, s\in\R,\,\, |i|\le \gamma^{'}t,\,\, t\ge T.
$$
Assume that there is $0<\gamma_0<\gamma^{'}$ such that \eqref{convergence-aux-eq2} does not hold. Then there are $\epsilon_0>0$,
$s_n\in\R$, $i_n\in\Z$, $t_n>0$ such that $|i_n|\le \gamma_0 t_n$, $t_n\to\infty$, and
\begin{equation}
\label{convergence-aux-eq3}
|u_{i_n}(s_n+t_n;s_n,u^0)-u^+(s_n+t_n)|\ge\epsilon_0.
\end{equation}

Let $\tilde  u^0=\{\tilde u^0_i\}$ and $\hat u^0=\{\hat u^0_i\}$, where $\tilde u^0_i=\frac{\delta_0}{2}$
 and $\hat u^0_i=\|u^0\|$ for all $i\in\Z$. By (1), there is $\tilde T\ge T$ such that
\begin{equation}
\label{convergence-aux-eq4}
|u_i(s+t;s,\tilde u^0)-u^+(s+t)|<\frac{\epsilon_0}{2}\quad \forall\,\, i\in\Z,\,\, s\in\R,\,\, t\ge \tilde T
\end{equation}
and
\begin{equation}
\label{convergence-aux-eq4-1}
u_i(s+t;s,u^0)\le u_i(s+t;s,\hat u^0)\le u^+(s+t)+\epsilon_0\quad \forall\,\, i\in\Z,\,\, s\in\R,\,\, t\ge \tilde T.
\end{equation}

Let $\tilde u^n=\{\tilde u^n_i\}$ be given by
$$
\tilde u^n_i=\begin{cases} \frac{\delta_0}{2}\quad \forall \,\, |i|\le (\gamma^{'}-\gamma_0)(t_n-\tilde T)\cr
0\quad \text{for otherwise}.
\end{cases}
$$
Then
$$
\lim_{n\to\infty} \tilde u^n_i= \tilde u^0_i\quad \text{locally uniformly}.
$$
By Proposition \ref{convergence-prop},
\begin{equation}
\label{convergence-aux-eq5}
\lim_{n\to\infty} \big(u_i(s_n+t_n;s_n+t_n-\tilde T,\tilde u^n)-u_i(s_n+t_n;s_n+t_n-\tilde T,\tilde u^0)\big)=0
\end{equation}
locally uniformly in $i\in\Z$.

Observe that
\begin{align*}
u_{i_n}(s_n+t_n;s_n,u^0)&=u_{i_n}(s_n+t_n;s_n+t_n-\tilde T,u(s_n+t_n-\tilde T;s_n,u^0))\\
&=u_0(s_n+t_n;s_n+t_n-\tilde T,u_{\cdot+i_n}(s_n+t_n-\tilde T;s_n,u^0))\\
&\ge u_0(s_n+t_n;s_n+t_n-\tilde T,\tilde u^n)\quad {\rm for}\,\, n\gg 1.
\end{align*}
This together with \eqref{convergence-aux-eq4}, \eqref{convergence-aux-eq4-1}, and \eqref{convergence-aux-eq5} implies that
\begin{equation}
\label{convergence-aux-eq6}
u^+(s_n+t_n)-\epsilon_0<u_{i_n}(s_n+t_n;s_n,u^0)< u^+(s_n+t_n)+\epsilon_0
\end{equation}
for $n\gg 1$, which contradicts to \eqref{convergence-aux-eq3}.
Hence \eqref{convergence-aux-eq2} holds.
\end{proof}

\section{Spreading speeds}

In this section, we investigate spreading speeds of \eqref{main-eqn} and prove Theorem \ref{main-thm1}.
First we present two lemmas.

For given $T>0$, recall that
$$
\bar f_T=\inf_{k\in \N}\frac{1}{T}\int_{(k-1)T}^{kT} f(\tau,0)d\tau.
$$
By Lemma \ref{technical-lemma}(3),
$$\bar{f}^+_{\inf}=\lim_{T\to\infty}\inf_{t\geq 0}\frac{1}{T}\int^{t+T}_tf(\tau,0)d\tau.$$
So we have

\begin{lemma}
\label{technical-lm2}
For given $\gamma^{'}<c_0^-$, there is $T>0$ such that
$$
\gamma^{'} <\inf_{\mu>0}\frac{e^{-\mu}+e^\mu-2+\bar f_T}{\mu}.
$$
\end{lemma}

 %The next result comes from \cite[Lemma 3.2, Remark 3.3]{NaRo12}.

%\begin{lemma}
%\label{technical-lm1}
%For given $T>0$,
%there is $A\in W^{1,\infty}((0,\infty))$ such that
%$$
%\essinf_{t\in (0,\infty)}\big(A^{'}(t)+f(t,0)\big)=\bar f_T.
%$$
%\end{lemma}

%\begin{lemma}
%\label{technical-lm3}
%There is $M>0$ such that
%$f(t,u)\ge f(t,0)-Mu$ for $u\ge 0$.
%\end{lemma}

\begin{lemma}
\label{technical-lm4}
For any given $ M>0$,
consider
\begin{equation}
\label{aux-eq1}
\dot u_i(t)=u_{i+1}(t)-2u_i(t)+u_{i-1}(t)+u_i(t)  \big(\bar f_T - Mu_i(t)\big).
\end{equation}
Let $[c_{*,T},c^{*,T}]$ be the spreading speed interval of \eqref{aux-eq1}. Then
$$
c_{*,T}=c^{*,T}=\inf_{\mu>0}\frac{e^{-\mu}+e^\mu-2+\bar f_T}{\mu}.
$$
\end{lemma}
\begin{proof}
See \cite[Theorem 2.3]{Sh09}.
\end{proof}

We now prove Theorem \ref{main-thm1}.

\begin{proof}[Proof of Theorem \ref{main-thm1}]

(1) First we prove that for any given $\gamma^{'}<c_0^-$ and $u^0\in l^\infty_0(\Z)$,
\begin{equation}
\label{lower-bound-eq1}
\liminf_{|i|\le \gamma^{'} t, t\to\infty} u_i(t;0,u^0)>0.
\end{equation}

For the given $\gamma^{'}<c_0^-$, let $T>0$ be as in Lemma \ref{technical-lm2} and $A(t)$ be as in Lemma \ref{technical-lemma}(2).
Put $v_i(t)=u_i(t;0,u^0)e^{A(t)}$. Then $v_i(t)$ is  absolutely continuous in and differentiable in $t\in [0,\infty)$  and  satisfies
\begin{align}
\label{aux-eq2}
\dot v_i(t)&=\dot u_i(t;0,u^0) e^{A(t)}+A^{'}(t) u_i(t;0,u^0)e^{A(t)}\nonumber\\
&=v_{i+1}(t)+v_{i-1}(t)-2v_i(t)+v_i(t)\big( f(t,u_i(t;0,u^0))+A^{'}(t)\big)\nonumber\\
&\ge v_{i+1}(t)+v_{i-1}(t)-2v_i(t)+v_i(t)\big( f(t,0)-\tilde M_0 u_i(t;0,u^0)+A^{'}(t)\big)\nonumber\\
&\ge v_{i+1}(t)+v_{i-1}(t)-2v_i(t)+v_i(t)\big(\bar f_T-\tilde M_0u_i(t;0,u^0)\big)\nonumber\\
&= v_{i+1}(t)+v_{i-1}(t)-2v_i(t)+v_i(t)\big(\bar f_T-\tilde M_0 e^{-A(t)}v_i(t)\big)\nonumber\\
&\ge  v_{i+1}(t)+v_{i-1}(t)-2v_i(t)+v_i(t)\big(\bar f_T-\tilde M v_i(t)\big)
\end{align}
for a.e. $t>0$, where $\tilde M=\tilde M_0 \sup\limits_{t>0} e^{-A(t)}$. By Lemmas \ref{technical-lm2} and \ref{technical-lm4},
$$
\liminf_{|i|\le \gamma^{'} t,t\to\infty} v_i(t)>0.
$$
This implies that \eqref{lower-bound-eq1} holds.

For any $\gamma<c_0^-$, let $\gamma^{'}\in (\gamma,c_0^-)$. By \eqref{lower-bound-eq1}
and  Theorem \ref{stable-thm}(2),
$$
\limsup_{|i|\le \gamma t, t\to\infty}|u_i(t;0,u^0)-u^+(t)|=0.
$$
 Thus $c_0^-\leq c_*$.

%Next we prove that for given $\gamma>0$, if for any $u^0\in l_0^\infty(\Z)$,
%$\liminf\limits_{|i|\le \gamma t,t\to\infty}u_i(t;0,u^0)>0$, then $\gamma< c_0^-$, which implies $c_0^-\geq c_*$.

%Assume that $\gamma>c_0^-$.  Note that there is $m>0$ such that
%$$
%f(t,u)\le f(t,0)-m u.
%$$
%Let $g(t,u)$ be defined by
%$$
%g(t,u)=\begin{cases} f(t,0)-mu\quad {\rm for}\quad t\ge 0\cr
%\bar f_{\inf}^+-mu\quad {\rm for}\quad t<0.
%\end{cases}
%$$
%Let $\gamma>\gamma^{'}>c_0^-$.
%Then there is a transition front solution of
%$$
%\dot u_i(t)=u_{i+1}(t)+u_{i-1}(t)-2u_i(t)+u_i(t) g(t,u_i(t))
%$$
%with speed $c(t)$ and $\bar c_{\inf}=\gamma^{'}$.
%\textbf{(W.S. Need to prove $\gamma\le \gamma^{'}$ here. This is a contradiction.)}
%It then follows that $c_*=c_0^-$.

Next we prove that for any $\gamma>c_0^+$ and $u^0\in l^\infty_0(\Z)$,
\begin{equation}
\label{upper-bound-eq1}
\limsup_{|i|\ge \gamma t,t\to\infty}u_i(t;0,u^0)=0.
\end{equation}

For the given $\gamma>c_0^+$, there is $\tilde T>0$ such that
\begin{equation}\label{greater}
\gamma>\inf_{\mu>0}\frac{e^{-\mu}+e^\mu-2+\tilde f_{\tilde T}}{\mu}\end{equation}
with $$\tilde f_{\tilde T}=\sup_{k\in \N}\frac{1}{\tilde T}\int_{(k-1)\tilde T}^{k\tilde T} f(\tau,0)d\tau.$$
Then by Lemma \ref{technical-lemma}(2) with $f(t,0)$ and $T$ replaced by $-f(t,0)$ and $\tilde T>0$, respectively, there is $\tilde A(t)\in W^{1,\infty}((0,\infty))$ such that
\begin{equation}
-\tilde f_{\tilde T}=\inf_{k\in \N}\frac{1}{\tilde T}\int_{(k-1)\tilde T}^{k\tilde T} \big(-f(\tau,0)\big)d\tau=
 \essinf_{t\in (0,\infty)}\big(-\tilde A^{'}(t)-f(t,0)\big).
\end{equation}
Put $\tilde v_i(t)=u_i(t;0,u^0)e^{\tilde A(t)}$. By (H2),
$f(t,u)\le f(t,0)-\tilde m_0 u$.
 Then $\tilde v_i(t)$ is absolutely continuous in $t\in [0,\infty)$  and satisfies
\begin{align}
\label{aux-eq2}
\dot {\tilde v}_i(t)&=\dot u_i(t;0,u^0) e^{\tilde A(t)}+\tilde A^{'}(t) u_i(t;0,u^0)e^{\tilde A(t)}\nonumber\\
&=\tilde v_{i+1}(t)+\tilde v_{i-1}(t)-2\tilde v_i(t)+\tilde v_i(t)\big( f(t,u_i(t;0,u^0))+\tilde A^{'}(t)\big)\nonumber\\
&\le \tilde v_{i+1}(t)+\tilde v_{i-1}(t)-2\tilde v_i(t)+\tilde v_i(t)\big( f(t,0)-\tilde m_0u_i(t;0,u^0)+\tilde A^{'}(t)\big)\nonumber\\
&\le \tilde v_{i+1}(t)+\tilde v_{i-1}(t)-2\tilde v_i(t)+\tilde v_i(t)\big(\tilde f_{\tilde T}-\tilde m_0u_i(t;0,u^0)\big)\nonumber\\
&= \tilde v_{i+1}(t)+\tilde v_{i-1}(t)-2\tilde v_i(t)+\tilde v_i(t)\big(\tilde f_{\tilde T}-\tilde m_0 e^{-\tilde A(t)}\tilde v_i(t)\big)\nonumber\\
&\le  \tilde v_{i+1}(t)+\tilde v_{i-1}(t)-2\tilde v_i(t)+\tilde v_i(t)\big(\tilde f_{\tilde T}-\tilde m \tilde v_i(t)\big)
\end{align}
for a.e. $t>0$,
where $\tilde m=\tilde m_0 \inf\limits_{t>0} e^{-\tilde A(t)}$. By Lemma \ref{technical-lm4} and  \eqref{greater},
$$
\limsup_{|i|\ge \gamma t,t\to\infty} \tilde v_i(t)=0.
$$
This implies that \eqref{upper-bound-eq1} holds. Thus $c_0^+\ge c^*$.

(2)  Note that from the proof of \cite[Lemma 3.2]{NaRo12}, we can also get that for given $T>0$,
there is $\hat A\in W^{1,\infty}(\R)$ such that
$$
\essinf_{t\in \R}\big(\hat A^{'}(t)+f(t,0)\big)=\inf_{k\in \Z}\frac{1}{T}\int_{(k-1)T}^{kT} f(\tau,0)d\tau.
$$ Then the results can be proved by the similar arguments as in (1).
\end{proof}

\section{Transition fronts}

In this section, we study transition fronts and prove Theorems \ref{main-thm2} and \ref{main-thm3}.
 We first prove some important lemmas.

%The following result is adopted from \cite[Lemma 3.2]{NaRo12} and will be used in the construction of sub-solution.
%\begin{lemma}\label{least mean}
%Let $B\in L^{\infty}(\R)$. Then
%$$\bar B_{\inf}=\sup_{A\in W^{1,\infty}(\R)} \essinf_{t\in\R}(A^{\prime}+B)(t).$$
%\end{lemma}

%Now we consider the problem
 %\begin{equation}\label{profile}
%\left\{\begin{split} &\pa_t\phi(x,t)-H\phi(x,t)-c(t)\pa_x\phi(x,t)=\phi(x,t) f(t,\phi(x,t)), x\in \R, t\in\R,
%\\
%&\lim_{x\to-\infty}(\phi(x,t)-u^+(t))=0\,\,\text{and}\,\,\lim_{x\to\infty}\phi(x,t)=0\,\,\text{uniformly in}\,\,t\in\R.
 %\end{split}
 %\right.
%\end{equation}

For given $\mu>0$, let
$$c(t;\mu)=\frac{e^{-\mu}+e^{\mu}-2+f(t,0)}{\mu}.
$$
Recall that
$$
 \tilde c_0^-=\inf_{\mu>0}\frac{e^{-\mu}+e^{\mu}-2+\bar f_{\inf}}{\mu}.
$$

\begin{lemma}
\label{mu-star-lemma}
 There is a unique $\mu^*>0$ such that
$$
\tilde c_0^-=\frac{e^{-\mu^*}+e^{\mu^*}-2+\bar f_{\inf}}{\mu^*}
$$
and
for any $\gamma>\tilde c_0^-$,
 the equation $\gamma=\frac{e^{-\mu}+e^{\mu}-2+\bar f_{\inf}}{\mu}$ has exactly two positive solutions for $\mu$.
 \end{lemma}

 \begin{proof}
Let $\chi_1(\mu)=\frac{e^{-\mu}+e^{\mu}-2+\bar f_{\inf}}{\mu}$ and $\chi_2(\mu)=\frac{\pa}{\pa \mu}(\mu\chi_1(\mu))$. Then
$$\frac{\pa \chi_2}{\pa \mu}(\mu)=e^\mu-e^{-\mu}>0,$$
$$\frac{\pa \chi_1}{\pa \mu}(\mu)=\frac{1}{\mu}[\chi_2(\mu)-\chi_1(\mu)],$$
$$\frac{\pa}{\pa\mu}(\mu^2\frac{\pa \chi_1}{\pa \mu}(\mu))=\mu\frac{\pa\chi_2}{\pa\mu}(\mu)>0 \mbox{ for }\mu>0.$$
Hence there is at most one $\mu>0$ such that
$$
\frac{\pa \chi_1}{\pa \mu}(\mu)=0.
$$
The lemma then follows from
$$\lim_{\mu\rightarrow +\infty}\chi_1(\mu)=+\infty,$$
and  $$\lim_{\mu\rightarrow 0^+}\chi_1(\mu)=+\infty \,\,(\mbox{by \eqref{assumption-eq}}).$$
\end{proof}

\begin{lemma}\label{subsuper-solu-lemma}
For any $\gamma>\tilde c_0^-$, let $0<\mu<\mu^*$ be such that $\chi_1(\mu)=\gamma$ and $c(t)=c(t;\mu)$. Then there are
 $\bar{\phi}(x,t)$ and $\underline{\phi}(x,t)$ satisfying the following properties.
 \begin{itemize}
\item[(1)]  $\bar{\phi}(x,t)$ and $\underline{\phi}(x,t)$ are continuous functions in $t\in\R$ and $x\in\R$,
$$
0<\underline{\phi}(x,t)<\bar{\phi}(x,t)\le u^+(t),\,\,\,\bar{\phi} \mbox{ is nonincreasing in } x\in \R,
$$
\item[(2)] For any $M\in\R$,
$$
\inf_{x\le M,t\in\R} \underline{\phi}(x,t)>0, \,\, \inf_{x\le M,t\in\R} \bar{\phi}(x,t)>0.
$$

\item[(3)] The limits
$$\lim_{x\to\infty}\frac{\bar \phi(x,t)}{e^{-\mu x}}=\lim_{x\to\infty}\frac{\underline \phi(x,t)}{e^{-\mu x}}=1$$
exist and are
uniform in $t\in\R$.

\item[(4)] Let
$$\bar{v}(\cdot,t)=\bar{\phi}(\cdot-\int^t_0c(\tau)d\tau,t)\quad {\rm and}\quad \underline{v}(\cdot,t)=\underline{\phi}(\cdot-\int^t_0c(\tau)d\tau,t).
$$
Then
$$
u(x,t;s,\bar v(\cdot,s))\le \bar v(x,t),\quad u(x,t;s,\underline v(\cdot,s))\ge \underline v(x,t)
$$
for all $x\in\R$ and $t\ge s$.
\end{itemize}
\end{lemma}

\begin{proof}
First of all, we may assume that $f(t,u)=f(t,0)$ for $u<0$. For otherwise, we may replace $f(t,u)$ be $\tilde f(t,u)$, where
$\tilde f(t,u)=f(t,u)$ for $u\ge 0$ and $\tilde f(t,u)=f(t,0)$ for $u<0$.

We first construct $\bar \phi(x,t)$ satisfying $\bar\phi(x,t)\le u^+(t)$ and (2)-(4).
Let  $\varphi(x)=e^{-\mu x}$. Then $\varphi (x)$ is a solution of the equation
\begin{align}
0&=\pa_t\varphi-H\varphi-c(t)\pa_x\varphi-f(t,0)\varphi\nonumber
\\&=\varphi[c(t)\mu-(e^{\mu}+e^{-\mu}-2+f(t,0))], \,\mbox{ for }     x\in \R, \,t\in\R.\nonumber
\end{align}
Let $\hat v(x,t)=\varphi(x-\int^t_0c(\tau)d\tau)=e^{-\mu (x-\int^t_0c(\tau)d\tau)}$. Then $\hat v(x,t)$ satisfies
$$\pa_t\hat v(x,t)=H\hat v(x,t)+f(t,0)\hat v(x,t)\ge H\hat v+\hat vf(t,\hat v), \quad\quad x\in \R,\, t\in\R.$$ Thus it is a super-solution of \eqref{main-eqn2}. Moreover, for any constant $C$,
$\hat u(x,t):=e^{Ct} \hat v(x,t)$  satisfies
$$
\pa_t \hat u(x,t)=(\pa_t\hat  v(x,t)+ C\hat v(x,t))e^{Ct}\ge H \hat u(x,t)+C \hat u(x,t)+\hat u(x,t) f(t,\hat v(x,t)),
$$
hence
$$
\hat u(x,t)\ge \hat u(x,s)+\int_s ^t \Big(H \hat u(x,\tau)+C \hat u(x,\tau)+
\hat u(x,\tau) f(\tau,\hat v(x,\tau))\Big) d\tau,
$$
and $\tilde u(x,t)=e^{Ct} u^+(t)$ satisfies
$$
\pa_t \tilde u(x,t)=(\pa_t  u^+(t) + C u^+(t))e^{Ct}= H \tilde u(x,t)+C \tilde u(x,t)+\tilde u(x,t) f(t,u^+(t)),
$$
hence
$$
\tilde u(x,t)= \tilde u(x,s)+\int_s ^t \Big(H \tilde u(x,\tau)+C \tilde u(x,\tau)+\tilde u(x,\tau) f(\tau, u^+(\tau))\Big) d\tau.
$$

Let
$$
\bar \phi(x,t)= \min\{\varphi(x),u^+(t)\}
 $$
 and
 $$
 \bar v(x,t)=\bar \phi(x-\int_0^t c(\tau)d\tau,t).
 $$
 It is clear that
 $$
 \bar \phi(x,t)\le u^+(t)\quad \forall\,\, x\in\R,\,\, t\in\R
 $$
 and that $\bar \phi(x,t)$ satisfies (2) and (3). We prove that $\bar\phi(x,t)$ also satisfies (4).

 Recall that $\bar v(x,t)=\bar\phi(x-\int_0^t c(\tau)d\tau,t)$. Note that  for any constant $C$, $u(x,t)=e^{Ct} \bar v(x,t)$ satisfies
$$
u(x,t)\ge u(x,s)+\int_s ^t \Big(H u(x,\tau)+C u(x,\tau)+u(x,\tau) f(\tau, \bar v(x,\tau))\Big) d\tau.
$$
Let
$w(x,t)=e^{Ct} \big(\bar v(x,t)-u(x,t;s,\bar v(\cdot;s))\big)$. Then
$$
w(x,t)\ge  w(x,s)+\int_s^t \Big(H w(x,\tau)+C w(x,\tau) +a(x,\tau) w(x,\tau)\Big)d\tau,
$$
where
$$a(x,\tau)=f(\tau,u(x,\tau;s,\bar v(\cdot,s)))+\bar v(x,\tau)\int^1_0f_u(\tau,r\bar v(x,\tau)+(1-r)u(x,\tau;s,\bar v(\cdot,s)))dr.
$$
Choose $C>0$ such that $C-2+a(x,t)>0$ for all $t\in\R$ and $x\in\R$.
By the arguments of Proposition \ref{comparison},
we have
$$
w(x,t)\ge w(x,s)=0,
$$
and hence
$$
u(x,t;s,\bar v(\cdot,s))\le \bar v(x,t)\quad \forall\,\, x\in\R,\,\, t\ge s.
$$
Hence $\bar\phi(x,t)$ also satisfies (4).

Next, we construct $\underline\phi(x,t)$ satisfying (1)-(4).
Let $\tilde M_0$ be as in (H2).
Let  $B(t)=-(e^{-\tilde{\mu}}+e^{\tilde{\mu}}-2)+c(t)\tilde{\mu}-f(t,0)$.
Note that
\begin{eqnarray*}
\bar B_{\inf}&=&-(e^{-\tilde{\mu}}+e^{\tilde{\mu}}-2)+\gamma \tilde{\mu}-\bar f_{\inf}\\&=&\tilde{\mu}(\gamma-\frac{e^{-\tilde{\mu}}+e^{\tilde{\mu}}-2+\bar f_{\inf}}{\tilde{\mu}}),
\end{eqnarray*}
thus we can choose $\tilde{\mu}\in(\mu,2\mu)$ such that $\bar B_{\inf}>0$.
Due to Lemma \ref{technical-lemma},  we can then find $A\in W^{1,\infty}(\R)$ such that $\essinf\limits_{t\in\R}(A^{\prime}+B)>0$.

Let $\psi(x,t)=e^{-\mu x}-e^{A(t)-\tilde\mu x}$.  Then for each $x$, $\psi(x,t)$ is absolutely continuous in $t$.
We claim that $\psi(x,t)$
 satisfies that for each $x\in\R$,
\begin{equation}\label{subsolution}
\pa_t\psi-H\psi-c(t)\pa_x\psi\leq \psi f(t,\psi)
\end{equation}
for   a.e. $t\in\R$.
Note that for each $x$, \begin{align}&\pa_t\psi-H\psi-c(t)\pa_x\psi-f(t,0)\psi\nonumber\\&=[-A^{\prime}(t)+e^{-\tilde{\mu}}
+e^{\tilde{\mu}}-2-c(t)\tilde{\mu}+f(t,0)]e^{A(t)-\tilde{\mu}x}\quad {\rm for}\quad a.e.\,\, t\in\R.\nonumber
\end{align}
Note also that
$$
f(t,0)-[A^{'}(t)-(e^{-\tilde \mu}+e^{\tilde \mu}-2)+c(t)\tilde \mu]<0\quad \forall\,\, t\in\R.
$$
For given $x\in\R$ and $t\in\R$ such that $\pa_t \psi(x,t)$ exists, if
 $\psi(x,t)\le 0$, then
\begin{align}&\pa_t\psi-H\psi-c(t)\pa_x\psi\nonumber\\&= f(t,0)\psi(x,t)+[-A^{\prime}(t)+e^{-\tilde{\mu}}
+e^{\tilde{\mu}}-2-c(t)\tilde{\mu}+f(t,0)]e^{A(t)-\tilde{\mu}x}\nonumber\\
&\le  f(t,0)\psi(x,t)\nonumber\\
& = \psi(x,t) f(t,\psi(x,t)).\nonumber
\end{align}
On the other hand, if $\psi(x,t)>0$, then   $x\geq (\tilde{\mu}-\mu)^{-1}A(t)$ and we have
$$\tilde M_0\psi^2e^{\tilde{\mu}x-A(t)}\leq \tilde M_0e^{(\tilde{\mu}-2\mu)x-A(t)}\leq \tilde M_0e^{-A(t)} \,\,\mbox{ for }\,x\geq\max(0,(\tilde{\mu}-\mu)^{-1}A(t)),\,t\in\R.$$
By adding a large constant $\al$ to $A(t)$, we have $(\tilde{\mu}-\mu)^{-1}A(t)>0$ and
\begin{equation}\label{inequality}
A^{\prime}(t)+B(t)\ge \tilde M_0\psi^2e^{\tilde{\mu}x-A(t)} \,\,\mbox{ for a.e. }\,t\in\R.
\end{equation}
This   implies
\begin{align}&\pa_t\psi-H\psi-c(t)\pa_x\psi\nonumber\\&= f(t,0)\psi(x,t)+[-A^{\prime}(t)+e^{-\tilde{\mu}}
+e^{\tilde{\mu}}-2-c(t)\tilde{\mu}+f(t,0)]e^{A(t)-\tilde{\mu}x}\nonumber\\
&\le  f(t,0)\psi(x,t)-\tilde M_0 \psi^2(x,t)\nonumber\\
& \le \psi(x,t) f(t,\psi(x,t)).\nonumber
\end{align}
Therefore, the claim holds.

Let $u^+_K(t)$ be the unique entire positive solution of
$$
\dot u= u(f(t,0)-K u).
$$
For $K\gg 1$, $\sup_{t\in\R} u^+_K(t)\ll \sup_{x\in\R,t\in\R} \psi (x,t)$, and $u^+_K(t)$ is a sub-solution of \eqref{main-eqn2}.
Note that for each $t$, there are $X_1(t)<X_2(t)$ such that
$u^+_K(t)=\psi(x-\int_0^t c(\tau)d\tau,t)$ for $x=X_i(t)$ ($i=1,2$),
$u^+_K(t)>\psi(x-\int_0^t c(\tau)d\tau,t)$ for $x<X_1(t)$ or $x>X_2(t)$, and
$u^+_K(t)<\psi(x-\int_0^t c(\tau)d\tau,t)$ for $X_1(t)<x<X_2(t)$. When $K\gg 1$, $X_2(t)-X_1(t)>1$.

Let $$
\underline \phi(x,t)=\begin{cases}
\psi(x,t),\quad x\ge X_1(t)-\int_0^t c(\tau)d\tau\cr
u^+_K(t),\quad x<X_1(t)-\int_0^t c(\tau)d\tau
\end{cases}
$$
and $\underline v(x,t)=\underline\phi(x-\int_0^ tc(\tau)d\tau,t)$.
By the similar arguments as in the construction of $\bar\phi$,  $u(x,t)= e^{Ct}\underline{v}(\cdot,t)$  satisfies
 $$
 u(x,t)\le  u(x,s)+\int_s^t \Big(H u(x,\tau)+C u(x,\tau)+u(x,\tau) f(\tau,\underline v(x,\tau))\Big)d\tau
 $$
 and
 $$
 u(x,t;s,\underline v(\cdot,s))\ge \underline v(x,t)
 $$
 for all $x\in\R$ and $t\ge s$. It is clear that $\underline \phi(x,t)$ satisfies (1)-(3). The lemma is thus proved.
\end{proof}

\begin{remark}
\label{super-solu-rk1}
(1) If $f(t,u)\equiv f(u)$, then $\bar \phi(x,t)\equiv \bar \phi(x)$.

\smallskip

(2) If $f(t,u)=f(t+T,u)$, then $\bar \phi(x,t+T)=\bar \phi(x,t)$.
\end{remark}

%\liminf_{x\leq cnT_0,n\to\infty}v(x,nT_0):=
%\liminf_{x\leq c^{\prime}t,t\to\infty}(v(x,t)-u^+(t)):=

\begin{lemma}\label{zero-number-lemma}
Suppose that $u(t)=\{u_j(t)\}\in l^\infty(\Z)$ and $v(t)=\{v_j(t)\}\in l^\infty(\Z)$ are nonnegative solutions of \eqref{main-eqn} on $[t_0,\infty)$
and $u_j(t_0)\not\equiv v_j(t_0)$.
If there is $j_0$ such that $u_{j}(t_0)\ge v_j(t_0)$ for $j\le j_0$ and $u_j(t_0)\le v_j(t_0)$ for $j>j_0$, then
for any $t>t_0$, there is $j_t\in \Z\cup\{-\infty,\infty\}$ such that $u_j(t)\ge v_j(t)$ for $j\le j_t$ and $u_j(t)\le v_j(t)$ for
$j>j_t$. Moreover, if  $u_{i_t}(t)=v_{i_t}(t)$ for some $i_t\in\Z$ and $t\in (t_0,\infty)$,
then $u_j(t)\ge v_j(t)$ for $j\le i_t$ and $u_j(t)\le v_j(t)$ for $j>i_t$.
\end{lemma}

\begin{proof}
If $u_j(t_0)=v_j(t_0)$ for all $j\le j_0$, then $u_j(t_0)\le v_j(t_0)$ for all $j\in\Z$. By Proposition \ref{comparison},
$u_j(t)< v_j(t)$ for all $t>t_0$ and $j\in Z$. The lemma then follows.

Similarly, if $u_j(t_0)=v_j(t_0)$ for all $j>j_0$, then $u_j(t)>v_j(t)$ for all $t>t_0$ and $j\in\Z$. The lemma also follows.

Assume that there are $j_1\le j_0$ and $j_2>j_0$ such that $u_{j_1}(t_0)>v_{j_1}(t_0)$ and $u_{j_2}(t_0)<v_{j_2}(t_0)$. Without loss of generality, we may assume that $j_1=j_0$. Then there is $\epsilon>0$ such that $u_{j_0}(t)>v_{j_0}(t)$ for $t_0\le t\le t_0+\epsilon$.
It follows from the arguments of Proposition \ref{comparison}(1) that
$$
u_j(t)>v_j(t) \quad {\rm for}\quad t\in (t_0,t_0+\epsilon],\,\,\, j\le j_0,
$$
and
$$
u_j(t)<v_j(t)\quad {\rm for}\quad t\in (t_0,t_0+\epsilon],\,\,\, j>j_0.
$$
By \cite[Lemma 4]{HaZi93}, for any $t>t_0$, there is $j_t\in \Z\cup\{-\infty,\infty\}$ such that $u_j(t)\ge v_j(t)$ for $j\le j_t$ and $u_j(t)\le v_j(t)$ for
$j>j_t$.

 Moreover, suppose that $u_{i_t}(t)=v_{i_t}(t)$ for some $t\in (t_0,\infty)$ and $i_t\in\Z$. We claim that $u_j(t)\ge v_j(t)$ for $j\le i_t$. For otherwise, assume that there is $j^*<i_t$ such that
$u_{j^*}(t)<v_{j^*}(t)$. Then $j_t\le j^*$ and $u_j(t)\le v_j(t)$ for $j^*\le j\le i_t$. Without loss of generality, we may assume that
$u_{i_t-1}(t)<v_{i_t-1}(t)$. Then $\dot u_{i_t}(t)<\dot v_{i_t}(t)$ and hence
$$
u_{i_t}(t-\epsilon)>v_{i_t}(t-\epsilon)\quad {\rm for}\quad 0<\epsilon\ll 1.
$$
Then $j_{t-\epsilon}>i_t$ and $u_{i_t-1}(t-\epsilon)>v_{i_t-1}(t-\epsilon)$ for $0<\epsilon\ll 1$. This implies that
$u_{i_t-1}(t)\ge v_{i_t-1}(t)$, which is a contradiction. Therefore, $u_j(t)\ge v_j(t)$ for $j\le i_t$.
Similarly, we can prove that $u_j(t)\le v_j(t)$ for $j> i_t$.
The lemma is thus proved.
\end{proof}

Next, we prove Theorem \ref{main-thm2}.

\begin{proof}[Proof of Theorem \ref{main-thm2}]

(1) Let $c(t)$, $\bar{v}$, $\underline{v}$ be as in Lemma \ref{subsuper-solu-lemma} with $\bar c_{\inf}=\gamma$. For $\tau \geq 0$, let $v^\tau(x,t)$, $t\geq -\tau$ be the solution of
\begin{equation}
\left\{\begin{split} &\pa_tv-Hv=v f(t,v), \quad\quad\quad\quad\, x\in \R,\, t> -\tau,
\\
&v(x,-\tau)=\bar{v}(x,-\tau),\quad\quad\quad\quad x\in\R
 \end{split}
 \right.
\end{equation}
and
let $v_\tau(x,t)$, $t\geq -\tau$ be the solution of
\begin{equation}
\left\{\begin{split} &\pa_tv-Hv=v f(t,v), \quad\quad\quad\quad\, x\in \R,\, t> -\tau,
\\
&v(x,-\tau)=\underline v (x,-\tau),\quad\quad\quad\quad x\in\R.
 \end{split}
 \right.
\end{equation}
By Lemma \ref{subsuper-solu-lemma}, we have
$$
\underline{v}(x,t)\leq v_\tau(x,t)\le v^\tau(x,t)\leq \bar{v}(x,t)\quad \forall\,\, x\in\R,\,\, t\ge -\tau.
$$ Since $\bar{v}(x, t)$ is nonincreasing in $x$,  $v^\tau(x,t)$ is nonincreasing in $x$.
Moreover,  if $\tau_2>\tau_1\geq 0$, then
$$
v_{\tau_1}(x,t)\le v_{\tau_2}(x,t)\le v^{\tau_2}(x,t)\leq v^{\tau_1}(x,t)
$$ for all $(x,t)\in\R\times[-\tau_1,\infty]$. Therefore $\lim_{\tau\to\infty}v_\tau(x,t)$ and
 $\lim_{\tau\to\infty}v^\tau(x,t)$ exist, and  $\lim_{\tau\to\infty}v_\tau(x,t)$ is lower-semicontinuous  and
 $\lim_{\tau\to\infty}v^\tau(x,t)$ is upper-semicontinuous.
Let  $v^\pm:\,\R\times\R\rightarrow\R$ be such that
$$\lim_{\tau\to\infty}v_\tau(x,t)=v^-(x,t)\,\,\mbox{ pointwise in }\,(x,t)\in\R\times\R$$
and
$$\lim_{\tau\to\infty}v^\tau(x,t)=v^+(x,t)\,\,\mbox{ pointwise in }\,(x,t)\in\R\times\R.$$
Then
\begin{equation}
\label{lower-upper-cont}
\liminf_{(x,t)\to (x_0,t_0)}v^-(x,t)\ge v^-(x_0,t_0),\quad \limsup_{(x,t)\to (x_0,t_0)}v^+(x,t)\le v^+(x_0,t_0)
\end{equation}
for any $(x_0,t_0)\in\R\times\R$.
Since we have for any $x\in \R$ and $t\geq -\tau$,
$$v_\tau(x,t)=v_\tau(x,0)+\int^t_0Hv_\tau(x,s)ds+\int^t_0 v_\tau f(s,v_\tau)ds$$
and
$$v^\tau(x,t)=v^\tau(x,0)+\int^t_0Hv^\tau(x,s)ds+\int^t_0 v^\tau f(s,v^\tau)ds.$$
Let $\tau\to\infty$, it follows from the dominated convergence theorem that for any $x\in\R$ and $t\in\R$,
$$v^\pm(x,t)=v^\pm(x,0)+\int^t_0Hv^\pm(x,s)ds+\int^t_0 v^\pm f(s,v^\pm)ds.$$
Then we find that $v^\pm (x,t)$ is differentiable in $t$ and satisfies
\begin{equation*} \pa_tv^\pm =Hv^\pm+v^\pm f(t,v^\pm), \quad\quad x\in \R,\, t\in\R.\end{equation*}

By Lemma \ref{subsuper-solu-lemma}, we have
$$
\underline v(x,t)\le v^-(x,t)\le v^+(x,t)\le\bar v(x,t)\quad \forall\,\, x\in\R,\,\, t\in\R.
$$
 This implies that $\rho(v^-(\cdot,t),v^+(\cdot,t))$ is bounded in $t$, and if $v^-(x,t)\not\equiv v^+(x,t)$, then
 $u_1(x,t)=v^-(x,t)$ and $u_2(x,t)=v^+(x,t)$ satisfy the conditions in Proposition \ref{part-metric}(3).
 Assume that $v^-(x,t)\not\equiv v^+(x,t)$. Then
by Proposition \ref{part-metric}(3), we have
that for any $\tau>0$ and $T\in\R$, there is $\delta>0$ such that
$$
\rho(v^-(\cdot,s+\tau),v^+(\cdot,s+\tau))<\rho(v^-(\cdot,s),v^+(\cdot,s))-\delta
$$
for $s\le T$. This implies that
\begin{align*}
\rho(v^-(\cdot,-n),v^+(\cdot,-n))&=\rho(v^-(\cdot,-(n+\tau)+\tau),v^+(\cdot,-(n+\tau)+\tau))\\
&<\rho(v^-(\cdot,-(n+\tau)),v^+(\cdot,-(n+\tau)))-\delta\\
&<\rho(v^-(\cdot,-(n+k\tau)),v^+(\cdot,-(n+k\tau)))-k\delta
\end{align*}
for all $n\in\N$ with $n\ge -T$ and $k\in \N$. This is a contradiction since $\rho(v^-(\cdot,-(n+k\tau)),v^+(\cdot,-(n+k\tau)))-k\delta<0$
for $k\gg 1$. Therefore,
$$
v^-(x,t)=v^+(x,t).
$$
Then by \eqref{lower-upper-cont},  $v(x,t):=v^+(x,t)$ is continuous in $x\in\R$ and $t\in\R$ and is nonincreasing in $x$.
 Moreover, we have $$\lim\limits_{x\to\infty}\frac{v(x+\int^t_0c(\tau)d\tau,t)}{e^{-\mu x}}=1\quad\mbox{ uniformly in }\,t\in\R.$$
 By Lemma \ref{subsuper-solu-lemma},
$$
\delta_0:=\inf_{x\le 0,t\in\R} v(x+\int_0^t c(\tau)d\tau,t)>0.
$$
By Theorem \ref{stable-thm}(1) and Proposition \ref{convergence-prop}, we have
 $$\lim\limits_{x\to-\infty}(v(x+\int^t_0c(\tau)d\tau,t)-u^+(t))=0\quad\mbox{ uniformly in }\,t\in\R.$$
Then we can get the desired function $\phi(x,t)=v(x+\int^t_0c(\tau)d\tau,t)$.

 (2) By Theorem \ref{main-thm1}, $\tilde c_*\ge \tilde c_0^-$.  Assume that $\tilde c_*>\tilde c_0^-$.
Fix  $\gamma$, $c^{'}$, and $c^{''}$ such that
$$
\tilde c_0^-<\gamma< c^{'}< c^{''}<\tilde c_*.
$$
Observe that $\tilde c_0^->0$. By Theorem \ref{main-thm1}, for any $u^0\in l_0^\infty(\Z)$,
\begin{equation}
\label{aux-new-eq1}
\limsup_{|i|\le  c^{''} t,t\to\infty} |u_i(t+s;s,u^0)-u^+(t+s)|=0
\end{equation}
uniformly in $s\in\R$. 

Let $\phi(x,t)=v(x+\int^t_0c(\tau)d\tau,t)$ be as in (1).
Let
$$
u^s_i=\phi(i+[\int_0^s c(\tau)d\tau],s)\quad \forall\,\, s\in\R.
$$
By (1), there is $u^0\in l_0^\infty(\Z)$ such that
\begin{equation*}
u^0\le u^s\quad \forall s\in\R.
\end{equation*}
Hence
\begin{equation*}
u_i(t;s,u^0)\le u_i(t;s,u^s)\quad \forall\,\, i\in\Z, \,\, s\in\R,\,\, t\ge s.
\end{equation*}
This together with \eqref{aux-new-eq1} implies that
\begin{equation}
\label{aux-new-eq2}
\limsup_{|i|\le c^{''}t,t\to\infty}|u_i(t+s;s,u^s)-u^+(t+s)|=0
\end{equation}
uniformly in $s\in\R$.

By (1) again,
$$
u_i(t;s,u^s)=\phi(i-\int_0^t c(\tau)d\tau+[\int_0^s c(\tau)d\tau)],t)\le \phi(i-\int_s^t c(\tau)d\tau-1,t)
$$ 
and then
\begin{equation}
\label{aux-new-eq3}
\limsup_{i\ge (c^{''}-c^{'}) t+\int_s^{t+s} c(\tau) d\tau,t\to\infty}u_i(t+s;s,u^s)=0
\end{equation}
uniformly in $s\in\R$. It follows from \eqref{aux-new-eq2} and \eqref{aux-new-eq3} that
$$
\bar c_{\inf}\ge c^{'}>\gamma,
$$
which is a contradiction.
Therefore,  $\tilde c_*=\tilde c_0^-$.

\end{proof}

\begin{remark}
\label{super-solu-rk2}
(1) If $f(t,u)\equiv f(u)$, then $\bar \phi(x, t)\equiv \bar \phi(x)$. We claim that
 $\phi(x,t)\equiv \phi(x)$. In fact, when $f(t,u)\equiv f(u)$, we have
 $$
 \int_{t_1}^{t_2} c(s)ds=\int_0^{t_2-t_1} c(s)ds\quad \forall\,\, t_1,t_2\in\R
 $$
 and
 $$
 u(x,t;s,u_0)=u(x,0;s-t,u_0)\quad \forall\,\, t\ge s,\,\, u_0\in l^{\infty,+}(\R).
 $$
 We then have
 \begin{align*}
 \phi(x,t)&=v^+(x+\int_0^t c(\tau)d\tau,t)\\
 &=\lim_{\tau\to\infty} v^\tau(x+\int_0^ t c(s)ds,t)\\
 &=\lim_{\tau\to \infty} u(x+\int_0^t c(s)ds,t;-\tau,\bar v(\cdot,-\tau))\\
 &=\lim_{\tau\to \infty} u(x,t;-\tau,\bar\phi(\cdot+\int_0^ tc(s)ds-\int_0^{-\tau} c(s)ds))\\
 &=\lim_{\tau\to\infty} u(x,t;-\tau,\bar \phi(\cdot-\int_0^{-(t+\tau)}c(s)ds))\\
 &=\lim_{\tau\to\infty} u(x,0;-(t+\tau),\bar\phi(\cdot-\int_0^{-(t+\tau)}c(s)ds))\\
 &= v^+(x,0)=\phi(x)\quad \forall\,\, t\in\R.
 \end{align*}
 The claim thus follows.

(2) If $f(t,u)=f(t+T,u)$, then $\bar \phi(x,t+T)=\bar\phi(x,t)$. We claim that
$\phi(x,t+T)=\phi(x,t)$. In fact, when $f(t+T,u)=f(t,u)$, we have
$$
\int_{T}^{t+T} c(s)ds=\int_0^t c(s) ds\quad \forall\,\, t\in\R
$$
and
$$
u(x,t+mT;nT,u_0)=u(x,t;(n-m)T,u_0)\quad \forall \,\, n,m\in\Z.
$$
We then have
\begin{align*}
\phi(x,t+T)&=v^+(x+\int_0^{t+T} c(s)ds,t+T)\\
&=\lim_{n\to\infty} v^{nT}(x+\int_0^{t+T} c(s)ds,t+T)\\
&=\lim_{n\to\infty} u(x+\int_0^{t+T} c(s)ds,t+T;-nT,\bar v(\cdot,-nT))\\
&=\lim_{n\to\infty} u(x+\int_0^{t+T} c(s)ds,t+T;-nT,\bar\phi(\cdot-\int_0^{-nT} c(s)ds,-nT))\\
&=\lim_{n\to\infty}u(x+\int_0^t c(s)ds,t+T;-nT,\bar\phi(\cdot+\int_0^T c(s)ds-\int_0^{-nT} c(s)ds,-nT))\\
&=\lim_{n\to\infty} u(x+\int_0^ tc(s)ds,t;-(n+1)T,\bar\phi(\cdot-\int_0^{-(n+1)T} c(s)ds, -(n+1)T))\\
&=v^+(x+\int_0^ tc(s)ds,t)=\phi(x,t)\quad\forall\,\,t\in\R.
\end{align*}
The claim thus also holds.
\end{remark}

We now prove Theorem \ref{main-thm3}.

\begin{proof}[Proof of Theorem \ref{main-thm3}]
First of all, let $\mu^*$ be as in Lemma \ref{mu-star-lemma}. For given $0<\mu\le \mu^*$, let
$c_\mu(t)=\frac{e^{-\mu}+e^{\mu}-2+f(t,0)}{\mu}$,
 $\varphi_\mu(x)=e^{-\mu x}$,
and $\bar\phi_\mu(x,t)=\min\{\varphi_\mu(x),u^+(t)\}$. Let
$$
\bar v_\mu(x,t)=\bar\phi_\mu(x-\int_0^t c_\mu(s)ds,t).
$$
By Theorem \ref{main-thm2}, for given $0<\mu<\mu^*$, for each $t\in\R$,
$$
v_\mu(x,t):=\lim_{\tau\to \infty}u(x,t;-\tau, \bar v_\mu(\cdot,-\tau))
$$
uniformly in $x\in \R$, and
\begin{equation}
\label{new-eq0}
\lim_{x\to -\infty} v_\mu(x+\int_0^t c_\mu(s)ds,t)=u^+(t),\quad \lim_{x\to\infty} v_\mu(x+\int_0^tc_\mu(s)ds,t)=0
\end{equation}
uniformly in $t\in\R$.

Next, for given $0<\mu<\mu^*$, $n\in\N$, and $t>-n$, let $x(\mu,t,n)$ be such that
$$
u(x(\mu,t,n),t;-n,\bar\phi_\mu(\cdot,-n))=u(0,t;-n,\bar\phi_\mu(\cdot+x(\mu,t,n),-n))=\frac{u^+(t)}{2}.
$$
Note that
$$
v_\mu(x+\int_0^t c_\mu(s)ds,t)=\lim_{n\to\infty} u(x,t;-n,\bar\phi_{\mu}(\cdot+\int_{-n}^t c_\mu(s)ds,-n)).
$$
By \eqref{new-eq0}, we have that there is $M>0$ such that for any $t\in\R$,
\begin{equation}
\label{new-eq0-1}
|x(\mu,t,n)-\int_{-n}^t c_\mu(s)ds|\le M\quad \forall \,\, n\gg 1.
\end{equation}
Moreover, for any $\epsilon>0$, there is $\tilde M_\epsilon>0$ such that for any $t\in\R$,
\begin{equation}
\label{new-eq0-2}
J^+(t,\mu,n)-J^-(t,\mu,n)<\tilde M_\epsilon\quad \forall\,\, n\gg 1,
\end{equation}
where $J^\pm(t,\mu,n)$ are such that
$$
u(j+x(\mu,t,n),t;-n,\bar\phi_\mu(\cdot,-n))\begin{cases} \ge u^+(t)-\epsilon\quad {\rm for}\quad j\le J^-(t,\mu,n)\cr
\le \epsilon\quad {\rm for}\quad {\rm for}\quad j\ge J^+(t,\mu,n).
\end{cases}
$$

\smallskip

Now, for given $n\in\N$ and $0<\mu<\mu^*$, there are $x(\mu^*,n)$ and $x(\mu,n):=x(\mu,0,n)$ such that
\begin{equation}
\label{new-eq1}
u(x(\mu^*,n),0;-n,\bar\phi_{\mu^*}(\cdot,-n))=\frac{u^+(0)}{2},\quad u(x(\mu,n),0;-n,\bar\phi_\mu(\cdot,-n))=\frac{u^+(0)}{2}.
\end{equation}
Note that
$$
u(x+x(\mu^*,n),t;-n,\bar \phi_{\mu^*}(\cdot,-n))=u(x,t;-n,\bar\phi_{\mu^*}(\cdot+x(\mu^*,n),-n))
$$
and
$$
u(x+x(\mu,n),t;-n,\bar\phi_\mu(\cdot,-n))=u(x,t;-n,\bar\phi_\mu(\cdot+x(\mu,n),-n))
$$
for $t\ge -n$ and $x\in\R$. Note also that there is $j_n\in\Z$ such that
\begin{equation}
\label{new-eq2}
\bar\phi_{\mu^*}(j+x(\mu^*,n),-n)\begin{cases} \ge \bar\phi_\mu(j+x(\mu,n),-n)\quad {\rm for}\quad j\le j_n\cr
< \bar\phi_\mu(j+x(\mu,n),-n)\quad {\rm for}\quad j>j_n.
\end{cases}
\end{equation}
By \eqref{new-eq1}, \eqref{new-eq2}, and Lemma \ref{zero-number-lemma},
\begin{equation}
\label{new-eq3}
u(j+x(\mu^*,n),0;-n,\bar\phi_{\mu^*}(\cdot,-n))\begin{cases} \ge u(j+x(\mu,n),0;-n,\bar\phi_\mu(\cdot,-n))\quad {\rm for}\quad j\le 0\cr
\le u(j+x(\mu,n),0;-n,\bar\phi_\mu(\cdot,-n))\quad {\rm for}\quad j> 0.
\end{cases}
\end{equation}
Note that there is $n_k\to\infty$ such that $\lim_{n_k\to\infty} u(j+x(\mu^*,n_k),0;-n_k,\bar\phi_{\mu^*}(\cdot,-n_k))$ exists for all
$j\in\Z$.
Without loss of generality, we may assume  that  $\lim_{n_k\to\infty} u(j+x(\mu^*,n_k),-m$; $-n_k,\bar\phi_{\mu^*}(\cdot,-n_k))$ exists for all
$j\in\Z$ and $m\in\N$.
Let
$$
u_j^{-m,*}=\lim_{n_k\to\infty} u(j+x(\mu^*,n_k),-m;-n_k,\bar\phi_{\mu^*}(\cdot,-n_k))\quad\forall j\in\Z.
$$
By Proposition \ref{convergence-prop},
$$
u(j,0;-m,u^{-m,*})=\lim_{n_k\to\infty}  u(j+x(\mu^*,n_k),0;-n_k,\bar\phi_{\mu^*}(\cdot,-n_k))=u_j^{0,*}\quad\forall j\in\Z,\,\, m\in\N,
$$
and then
$$
u(j,t;0,u^{0,*})=\lim_{n_k\to \infty}  u(j+x(\mu^*,n_k),t;-n_k,\bar\phi_{\mu^*}(\cdot,-n_k))\quad\forall j\in\Z,\quad t\in\R.
$$
Hence $u(j,t;0,u^{0,*})$ is an entire solution of \eqref{main-eqn}.
It is clear that $u(j,t;0,u^{0,*})$ is nonincreasing in $j\in\Z$.

We claim that $u(j,t;0,u^{0,*})$ is a transition front solution satisfying the properties in Theorem \ref{main-thm3}.

To prove the claim,  we first prove that
 \begin{equation}
\label{new-eq4}
\lim_{j\to -\infty}u(j,t;0,u^{0,*})=u^+(t),\quad \lim_{j\to \infty}u(j,t;0,u^{0,*})=0\quad \forall\,\, t\in\R.
\end{equation}
 Note that without loss of generality, we may assume $\lim_{n_k\to\infty}u(j,0$; $-n_k,\bar\phi_\mu(\cdot+x(\mu,n_k),-n_k))$
exists for all $j\in\Z$. Let
$$
\tilde v^\mu_j=\lim_{n_k\to\infty}u(j,0;-n_k,\bar\phi_\mu(\cdot+x(\mu,n_k),-n_k))\quad\forall\,\, j\in\Z.
$$
By  Proposition \ref{convergence-prop} again,
$$
\tilde v_\mu(j,t):=u(j,t;0,\tilde v^\mu)=\lim_{n_k\to\infty}u(j,t;-n_k,\bar\phi_\mu(\cdot+x(\mu,n_k),-n_k))\quad\forall\,\, j\in\Z,\,\,\, t\in\R.
$$
Note that
$$
v_\mu(x,t)=\lim_{n\to\infty} u(x,t;-n,\bar\phi_\mu(\cdot-\int_0^{-n}c_\mu(\tau)d\tau,-n))\quad \forall \,\, x\in\R.
$$
By \eqref{new-eq1},
$$
\tilde v_\mu(0,0)=\frac{u^+(0)}{2}.
$$
By \eqref{new-eq0-1}, $x(\mu,n_k)+\int_0^{-n_k} c_\mu(\tau)d\tau$ is bounded. Hence
$$
\lim_{j\to -\infty} \tilde v_\mu(j,0)=u^+(0),\quad \lim_{j\to\infty} \tilde v_ \mu(j,0)=0.
$$
It then follows from the monotonicity of $\tilde v_\mu(j,t)$ in $j\in\Z$ that
$$
\lim_{j\to -\infty} \tilde v_\mu(j,t)=u^+(t),\quad \lim_{j\to\infty} \tilde v_ \mu(j,t)=0\quad \forall \,\, t\in\R.
$$
This together with \eqref{new-eq3} implies that \eqref{new-eq4}.

Next, let $J(t)\in\Z$ be such that
$$
u(j,t;0,u^{0,*})\begin{cases} \ge \frac{u^+(t)}{2}\quad {\rm for}\quad j\le J(t)\cr
<\frac{u^+(t)}{2}\quad {\rm for}\quad j>J(t).
\end{cases}
$$
By \eqref{new-eq4}, $J(t)$ is well defined for each $t\in\R$. We prove that
\begin{equation}
\label{new-eq5}
\lim_{j\to -\infty}u(j+J(t),t;0,u^{0,*})=u^+(t),\quad \lim_{j\to \infty}u(j+J(t),t;0,u^{0,*})=0
\end{equation}
uniformly in $t\in\R$.
To end this, fix $t\in\R$. Note that there is
$x^*_{t,n_k}\in\R$ such that
$$
u(x^*_{t,n_k}+J(t)+x(\mu^*,n_k),t;-n_k,\bar\phi_{\mu^*}(\cdot,-n_k))=\frac{u^+(t)}{2}
$$
Then
$$
u(1+J(t)+x(\mu^*,n_k),t;-n_k,\bar \phi_{\mu^*}(\cdot,-n_k))<u(x^*_{t,n_k}+J(t)+x(\mu^*,n_k),t;-n_k,\bar\phi_{\mu^*}(\cdot,-n_k))
$$
for $k\gg 1$. Recall that
$$
u(x(\mu,t,n_k),t;-n_k,\bar\phi_\mu(\cdot,-n_k))=\frac{u^+(t)}{2}.
$$
By the similar arguments for \eqref{new-eq3},
\begin{align}
\label{new-eq6}
&u(j+x^*_{t,n_k}+J(t)+x(\mu^*,n_k),t;-n_k,\bar\phi_{\mu^*}(\cdot,-n_k))\nonumber\\
&\begin{cases}\ge u(j+x(\mu,t,n_k),t;-n_k,\bar\phi_\mu(\cdot,-n_k))\quad\forall\,\, j\le 0\cr
\le u(j+x(\mu,t,n_k),t;-n_k,\bar\phi_\mu(\cdot,-n_k))\quad \forall\,\, j>0.
\end{cases}
\end{align}
This together with \eqref{new-eq0-2} implies that  there is $J^*$ such that
$$
u(J^*+x^*_{t,n_k}+J(t)+x(\mu^*,n_k),t;-n_k,\bar\phi_{\mu^*}(\cdot,-n_k))<\frac{u^+(t)}{4}\quad \forall \,\, k\gg 1.
$$
It then follows that
$$
u(J(t)+x(\mu^*,n_k),t;-n_k,\bar \phi_{\mu^*}(\cdot,-n_k))>u(J^*+x^*_{t,n_k}+J(t)+x(\mu^*,n_k),t;-n_k,\bar\phi_{\mu^*}(\cdot,-n_k))
$$
for $k\gg 1$. Therefore
$$
x^*_{t,n_k}+J(t)-1+x(\mu^*,n_k)<J(t)+x(\mu^*,n_k)<J^*+x^*_{t,n_k}+J(t)+x(\mu^*,n_k)
$$
and then
\begin{align}
\label{new-eq7}
&u(j+J^*+x^*_{t,n_k}+J(t)+x(\mu^*,n_k),t;-n_k,\bar\phi_{\mu^*}(\cdot,-n_k))\nonumber\\
&\le u(j+J(t)+x(\mu^*,n_k),t;-n_k,\bar\phi_{\mu^*}(\cdot,-n_k))\nonumber\\
&\le u(j+x^*_{t,n_k}+J(t)-1+x(\mu^*,n_k),t;-n_k,\bar \phi_{\mu^*},-n_k))
\end{align}
for $k\gg 1$.
This together with \eqref{new-eq0-2} and \eqref{new-eq6} implies \eqref{new-eq5}.

We now prove that
\begin{equation}
\label{new-eq8}
\liminf_{t-s\to\infty}\frac{J(t)-J(s)}{t-s}=\tilde c_0^-.
\end{equation}
By Theorem \ref{main-thm1}, we have
$$
\liminf_{t-s\to\infty}\frac{J(t)-J(s)}{t-s}\ge \tilde c_0^-.
$$
Fix $s<t$ and $0<\mu<\mu^*$. By \eqref{new-eq6} and \eqref{new-eq7},
\begin{equation}
\label{new-eq9}
u(j+J(s)+x(\mu^*,n_k),s;-n_k,\bar\phi_{\mu^*}(\cdot,-n_k))\le  u(j-1+x(\mu,s,n_k),s;-n_k,\bar\phi_\mu(\cdot,-n_k))
\end{equation}
for $j\ge 2$ and $k\gg 1$. By \eqref{new-eq0-1}, without loss of generality, we may assume that there is $j_s$, which is bounded in $s$ such that
$$
\lim_{n_k\to\infty} u(j-1+x(\mu,s,n_k),s;-n_k,\bar\phi_\mu(\cdot,-n_k))=v_\mu (j+j_s+\int_0^s c_\mu(r)dr,s)\quad \forall\,\, j\in\Z.
$$
This together with \eqref{new-eq9} implies that
\begin{equation}
\label{new-eq10}
u(j+J(s),s;0,u^{0,*})\le v_\mu (j+j_s+\int_0^s c_\mu(r)dr,s)\quad \forall\,\, j\ge 2.
\end{equation}
Note that
$$
\inf_{s\in\R} v_\mu(1+j_s+\int_0^s c_\mu(r)dr,s)>0.
$$
Let
$$
\kappa =\begin{cases}\inf_{j\le 0}\frac{v_\mu(j_s+\int_0^s c_\mu(r)dr,s)}{u(j+J(s),s;0,u^{0,*})}\quad {\rm if}\quad u(1+J(s),s;0,u^{0,*})\le
v_\mu(1+j_s+\int_0^s c_\mu(r)dr,s)\cr
\cr
\inf_{j\le 1} \frac{v_\mu(1+j_s+\int_0^s c_\mu(r)dr,s)}{u(j+J(s),s;0,u^{0,*})}\quad {\rm if}\quad u(1+J(s),s;0,u^{0,*})>
v_\mu(1+j_s+\int_0^s c_\mu(r)dr,s).
\end{cases}
$$
Then $0<\kappa<1$ and
$$
\kappa u(j+J(s),s;0,u^{0,*})\le v_\mu (j+j_s+\int_0^s c_\mu(r)dr,s)\quad \forall j\in\Z.
$$
Observe that  for $t\ge s$,
$$
\kappa u(j+J(s),t;0,u^{0,*})\le u(j,t;s, \kappa  u(\cdot+J(s),s;0,u^{0,*}))\le v_\mu(j+j_s+\int_0^s c_\mu(r)dr, t)\quad \forall\, j\in\Z.
$$
This implies that there is $L>0$ independent of $s$ and $t$  such that
$$
J(t)-J(s)\le \int_s^t c_\mu(r)dr+L.
$$
It then follows that
$$
\liminf_{t-s\to\infty}\frac{J(t)-J(s)}{t-s}\le \liminf_{t-s\to\infty}\frac{\int_s^ tc_\mu(r)dr}{t-s}
$$
for any $0<\mu<\mu^*$. Then by Theorem \ref{main-thm2},
$$
\liminf_{t-s\to\infty}\frac{J(t)-J(s)}{t-s}\le \tilde c_0^-.
$$
Therefore, \eqref{new-eq8} holds and Theorem \ref{main-thm3} is thus proved.
\end{proof}

%%%%%%%%%%%%%%%%%%%%%%%%%%%%%%%%%%%%%%%

\bibliographystyle{amsplain}

\end{document}